\documentclass[bj]{imsart}

\RequirePackage[OT1]{fontenc}
\usepackage{amsmath,amssymb,amsthm}
\usepackage{natbib}
\usepackage{epsfig,graphicx}
\usepackage{comment,enumerate,enumitem}
\usepackage{color}
\usepackage{srcltx}
\usepackage[mathscr]{eucal}
\usepackage[math]{easyeqn}
\usepackage{etoolbox}
\usepackage{hyperref}
\RequirePackage[colorlinks,citecolor=blue,urlcolor=blue]{hyperref}
\usepackage{nameref}
\startlocaldefs
\numberwithin{equation}{section}
\theoremstyle{plain}
\newtheorem{theorem}{Theorem}[section]
\newtheorem{lemma}[theorem]{Lemma}
\theoremstyle{remark}
\newtheorem{remark}{Remark}[section]
\newtheorem{corollary}[theorem]{Corollary}
\endlocaldefs

\renewcommand{\(}{$\,}
\renewcommand{\)}{\,$}

\def\nquad{\hspace{-1cm}}
\def\eqdef{\stackrel{\operatorname{def}}{=}}

\newcommand{\cc}[1]{\mathscr{#1}}
\newcommand{\bb}[1]{\boldsymbol{#1}}

\renewcommand{\bar}[1]{\overline{#1}}

\renewcommand{\tilde}[1]{\widetilde{#1}}

\renewcommand{\Gamma}{\varGamma}
\renewcommand{\Pi}{\varPi}
\renewcommand{\Sigma}{\varSigma}
\renewcommand{\Delta}{\varDelta}
\renewcommand{\Lambda}{\varLambda}
\renewcommand{\Psi}{\varPsi}
\renewcommand{\Phi}{\varPhi}
\renewcommand{\Theta}{\varTheta}
\renewcommand{\Omega}{\varOmega}
\renewcommand{\Xi}{\varXi}
\renewcommand{\Upsilon}{\varUpsilon}

\def\argmax{\operatornamewithlimits{argmax}}

\def\av{\bb{a}}

\def\fv{\bb{f}}

\def\Yv{\bb{Y}}

\def\betav{\bb{\beta}}

\def\epsv{\bb{\varepsilon}}
\def\etav{\bb{\eta}}

\def\lambdav{\bb{\lambda}}
\def\muv{\bb{\mu}}

\def\thetav{\bb{\theta}}

\def\xiv{\bb{\xi}}

\def\Psiv{\bb{\Psi}}

\usepackage{color}

\definecolor{blue(pigment)}{rgb}{0.2, 0.2, 0.6}
\definecolor{ultramarine}{rgb}{0.07, 0.04, 0.56}
\definecolor{darkspringgreen}{rgb}{0.09, 0.45, 0.27}
\definecolor{hookersgreen}{rgb}{0.0, 0.44, 0.0}
\definecolor{plum(traditional)}{rgb}{0.56, 0.27, 0.52}
\definecolor{purple(html/css)}{rgb}{0.5, 0.0, 0.5}
\definecolor{magenta(dye)}{rgb}{0.79, 0.08, 0.48}

\def\D{\bb{D}}

\def\R{I\!\!R}
\def\E{I\!\!E}

\def\P{I\!\!P}
\def\C{\mathbb{C}}

\def\kappa{\varkappa}

\def\diag{\operatorname{diag}}

\def\Exp{\operatorname{Exp}}
\def\Fr{\operatorname{Fr}}

\def\ND{\mathcal{N}}

\def\Var{\operatorname{Var}}

\def\T{\top}
\def\tr{\operatorname{tr}}
\def\diag{\operatorname{diag}}

\def\TV{\operatorname{TV}}

\def\CONST{\mathtt{C} \hspace{0.1em}}
\def\cond{\, \big| \,}

\def\nsize{{n}}

\def\sumi{\sum_{i=1}^{\nsize}}

\def\Id{I\!\!\!I}

\def\alp{\alpha}

\def\BB{I\!\!B}     
\def\BB{B}

\def\CS{\cc{E}}
\def\DP{D}

\def\DPGP{\DP_{\GP}}

\def\dimp{p}

\def\eps{\epsilon}			
\def\eps{\varepsilon}

\def\fvs{\fv}

\def\fvs{\fv^{*}}

\def\GP{G}

\def\LT{L}
\def\LGP{\LT_{\GP}}

\def\Lb{L^{\sbt}}

\def\lambdav{\bb{\lambda}}

\def\Pb{\P^{\sbt}}

\def\prior{\pi}		
\def\prior{\Pi}

\def\QL{W}

\def\qq{q}

\def\rr{\mathtt{r}}

\def\Sigmab{\Sigma^{\sbt}}

\def\sbt{\hspace{1pt} \flat}

\def\thetav{\bb{\theta}}
\def\thetavs{\thetav^{*}}
\def\thetavc{\thetav'}

\def\thetavb{\breve{\thetav}}

\def\vtheta{\vartheta}
\def\vthetav{\bb{\vtheta}}

\def\wb{w^{\sbt}}		

\def\WF{\bb{A}}

\def\xivb{\xiv^{\sbt}}

\def\ZZ{\mathtt{z}}		
\def\ZZ{\mathfrak{z}}

\def\ZZbt{\ZZ^{\sbt}}

\def \HM{\mathbb{H}}
\def \A{\mathbf{A}}
\def \D{\bb{D}}
\def \BB{\mathbf{B}}
\def \V{\mathbf{V}}
\def \G{\mathbf{G}}
\def \OP{\mathbf{P}}

\def \ee{\mathbf{e}}

\def\Sigmafr{\Sigma}
\def\CONSTdlt{\kappa}
\def\xivfr{\xiv}
\def\Zbar{\bar{Z}}
\def\tsqu{t}
\def\Inta{H}
\def\avc{\bar{\av}}
\def\ac{\bar{a}}
\def\Frobg{\Lambda}
\def\Cclass{\mathbf{C}}

\newcommand\myeq{\mathrel{\overset{\makebox[0pt]{\mbox{\normalfont\tiny\sffamily d}}}{=}}}

\begin{document}
\begin{frontmatter}

\title{Large ball probabilities, Gaussian comparison and anti-concentration}
\runtitle{Large ball probability}

\begin{aug}
	\author{\fnms{Friedrich} 
		\snm{G{\"o}tze}
		\thanksref{a}
		\ead[label=e1]{goetze@math.uni-bielefeld.de}},
	\author{\fnms{Alexey} 
		\snm{Naumov}
		\thanksref{b,d,e}
		\ead[label=e2]{anaumov@hse.ru, vulyanov@hse.ru}},
	\author{\fnms{Vladimir Spokoiny}
		\thanksref{b,d,e,f}
		\ead[label=e3]{spokoiny@wias-berlin.de}}
	\and
	\author{\fnms{Vladimir} 
		\snm{Ulyanov}
		\thanksref{b,c}%
		\ead[label=e4]{vulyanov@cs.msu.ru}}.
	
	\address[a]{Faculty of Mathematics, Bielefeld University, P. O. Box 10 01 31, 33501, Bielefeld, Germany.
	\printead{e1}}
	
	\address[b]{National Research University Higher School of Economics, 20 Myasnitskaya ulitsa, 101000, Moscow, Russia
	\printead{e2}}
	
	\address[c]{Moscow State University, 
		Leninskie Gory, 1, 
		Moscow, Russia
	\printead{e4}}
	\runauthor{F. G{\"o}tze et al.}
	
	\address[d]{Skolkovo Institute of Science and Technology (Skoltech), Skolkovo Innovation Center, Building 3, 143026, Moscow , Russia}
	
	\address[e]{Institute for Information Transmission Problems RAS, Bolshoy Karetny per. 19, bld.1, 127051, Moscow, Russia.}
	
	\address[f]{Weierstrass Institute, Mohrenstr. 39, 10117 Berlin, Germany
	\printead{e3}}
	
	
\end{aug}

\begin{abstract}
We derive tight non-asymptotic bounds for the Kolmogorov distance between the probabilities of two Gaussian elements to hit a ball in a Hilbert space. 
The key property of these bounds is that they are dimension-free and depend on the nuclear (Schatten-one) norm of the difference
between the covariance operators of the elements and on the norm of the mean shift. 
The obtained bounds significantly improve the bound based on Pinsker's inequality via the Kullback-Leibler divergence.
We also establish an anti-concentration bound for a squared norm of a non-centered Gaussian element in Hilbert space. 
The paper presents a number of examples motivating our results and  applications of the obtained bounds to statistical inference and to high-dimensional CLT. 
\end{abstract}


\end{frontmatter}

\section{Introduction}

In many statistical and probabilistic applications one faces the problem 
to evaluate how the probability of a ball under a Gaussian measure is affected, 
if the mean and the covariance operators of this Gaussian measure are 
slightly changed. 
Below we present particular examples motivating our results when such 
``large ball probability'' problem naturally arises,
including bootstrap validation, Bayesian inference, high-dimensional CLT.
This paper presents sharp bounds for the Kolmogorov distance between the probabilities of two Gaussian elements to hit a ball in a Hilbert space. 
The key property of these bounds is that they are dimension-free and depend on the nuclear (Schatten-one) norm of the difference between the covariance operators of the elements. 
We  also state a tight dimension free anti-concentration bound for a squared norm of a Gaussian element in Hilbert space which refines the well known results on the density
of a chi-squared distribution; see Theorem~\ref{band of GE}. 

Section~\ref{Sapplexamples} presents some 
application examples where the ``large ball probability'' issue naturally 
arises and explains how the new bounds of this paper can be used 
to improve the existing results.
The key observation behind the improvement is that in all mentioned examples 
we only need to know the properties of Gaussian measures on a class of balls. It means, in particular, that we would like to compare two Gaussian measures on the class of balls instead on  the class of all measurable sets. 
The latter can be upperbounded by general Pinsker's inequality via the Kullback--Leibler divergence. In case of Gaussian measures this divergence can be expressed explicitly in terms of parameters of the underlying measures, see e.g. 
\cite{SpZh2014}. 
However, the obtained bound involves the inverse of the covariance operators
of the considered Gaussian measures.
In particularly, small eigenvalues have the largest impact which is 
contra-intuitive if a probability of a ball is considered.
Our bounds only involve the operator and Frobenius norms of the related 
covariance operators and apply even in Hilbert space setup.

The proofs of the present optimal results are based in particular on Theorem~\ref{l: density est 2} below. 
This theorem gives sharp upper bounds for a probability density function  \(p_{\xiv}(x, \av)\) of \(\| \xiv - \av \|^{2}\), where \(\xiv\) is a Gaussian element with zero mean in a Hilbert space \(\HM\) with norm \(\| \cdot\| \) and \(\av\in \HM \). It is well known that \(p_{\xiv}(x, \av)\) can be considered as a density function of a weighted sum of non-central \( \chi^{2}\) distributions.   
An explicit but cumbersome  representation for \(p_{\xiv}(x,\av)\) in finite dimensional space \(\HM\) is available (see e.g. Section 18 in \cite{JKotzB1994}). 
However, it involves some special characteristics of the related Gaussian measure which makes it hard to use 
in specific situations.
Our results from Theorem~\ref{l: density est 2} and by Lemma~\ref{l: density est} 
are much more transparent and provide 
sharp uniform and non-uniform upper bounds on the underlying density respectively. 

One can even get two-sided bounds for \(p_{\xiv}(x, \av)\) but under additional conditions, see e.g. \cite{christoph1996}.  
Asymptotic properties of  \(p_{\xiv}(x, \av)\),
small balls probabilities \( \P\bigl(\| \xiv - a\| \leq \varepsilon \bigr) \),
or large deviation bounds \( \P\bigl(\| \xiv \| \geq 1/\varepsilon \bigr) \)
for small \( \varepsilon \)
can be found e.g. in
\cite{Bogach98}, \cite{LedouxTalag2002}, \cite{LiSh05},  
\cite{Lifshits2012} and \cite{Yurin95}.

The paper is organized as follows: 
a list of examples motivating our results and possible applications are given in Section~\ref{Sapplexamples}. 
Section~\ref{SmainresGC} collects the main results.
The proofs are given in Section~\ref{SproofsGC}.  
Some technical results and non-uniform upper bounds for \(p_{\xiv}(x, \av)\) are presented in the appendix.

\subsection{Application examples}
\label{Sapplexamples}
This section collects some examples where the developed results seem to be very useful.
\subsubsection{Bootstrap validity for the MLE}
\label{motivation 1}
Consider an independent sample \( \Yv = (Y_{1},\ldots, Y_{n})^{\T} \) with a joint distribution 
\( \P = \prod_{i=1,\ldots,n} P_{i} \).
The parametric maximum likelihood approach assumes that \( \P \) belongs to a given parametric family 
\( \bigl( \P_{\thetav} \, , \thetav \in \Theta \subseteq \R^{\dimp} \bigr) \) dominated by a measure 
\( \muv \), that is, \( \P = \P_{\thetavs} \) for \( \thetavs \in \Theta \).
The corresponding log-likelihood function can be written as a sum of marginal log-likelihoods 
\( \ell_{i}(Y_{i},\thetav) \):
\begin{EQA}
	L(\thetav)
	& \eqdef &
	\log \frac{d\P_{\thetav}}{d\muv}(\Yv)
	=
	\sumi \ell_{i}(Y_{i},\thetav) ,
	\qquad
	\ell_{i}(Y_{i},\thetav)
	=
	\log \frac{dP_{i,\thetav}}{d\mu_{i}}(Y_{i}). 
	\label{LYtsieliYit}
\end{EQA}
The MLE \( \tilde{\thetav} \) of the true parameter \( \thetavs \) is defined as 
the point of maximum of \( L(\thetav) \):
\begin{EQA}
	\tilde{\thetav}
	& \eqdef &
	\argmax_{\thetav \in \Theta} L(\thetav) ,
	\qquad
	L(\tilde{\thetav})
	\eqdef
	\max_{\thetav \in \Theta} L(\thetav). 
	\label{ttvdefamtLYt}
\end{EQA}
If the parametric assumption is misspecified, the target \( \thetavs \) is defined as the best parametric fit:
\begin{EQA}
	\thetavs
	& \eqdef &
	\argmax_{\thetav \in \Theta} \E L(\thetav) .
	\label{tvsdefamtLYt}
\end{EQA}
The likelihood based confidence set \( \CS(\ZZ) \) for the target parameter \( \thetavs \) is given by
\begin{EQA}
	\CS(\ZZ)
	& \eqdef &
	\bigl\{ \thetav \colon L(\tilde{\thetav}) - L(\thetav) \leq \ZZ \bigr\}.
	\label{CSzdeftLttLtz}
\end{EQA}
The value \( \ZZ \) should be selected to ensure the prescribed coverage probability \( 1 - \alpha \):
\begin{EQA}
	\P\bigl( \thetavs \not\in \CS(\ZZ) \bigr)
	& \leq &
	\alpha .
	\label{PtsniEzalp}
\end{EQA}
However, it depends on the unknown measure \( \P \).
The bootstrap approach is a resampling technique based on  
the conditional distribution of the reweighted log-likelihood \( \Lb(\thetav) \)
\begin{EQA}
	\Lb(\thetav)
	&=&
	\sumi \ell_{i}(Y_{i},\thetav) \wb_{i}
	\label{LbYtsieliYit}
\end{EQA}
with i.i.d. random weights \( \wb_{i} \) given the data \( \Yv \).
Below we assume that \( \wb_{i} \sim \ND(1,1) \). 
The bootstrap confidence set is defined as 
\begin{EQA}
	\CS^{\sbt}(\ZZ)
	& \eqdef &
	\bigl\{ \thetav \colon \sup_{\thetavc \in \Theta} \Lb(\thetavc) - \Lb(\thetav) \leq \ZZ \bigr\}.
	\label{CSzdeftLttLtz}
\end{EQA}
The bootstrap distribution is perfectly known and the bootstrap quantile \( \ZZbt \) is defined by
the condition
\begin{EQA}
	\Pb\bigl( \tilde{\thetav} \not\in \CS^{\sbt}(\ZZbt) \bigr)
	& = &
	\Pb\Bigl( \sup_{\thetav \in \Theta} \Lb(\thetav) - \Lb(\tilde{\thetav}) > \ZZbt \Bigr)
	=
	\alpha .
	\label{PbttniCSbZZ}
\end{EQA}
The bootstrap approach suggests to use \( \ZZbt \) in place of \( \ZZ \) to ensure \eqref{PtsniEzalp}
in an asymptotic sense. 
Bootstrap consistency means that for \( n \) large
\begin{EQA}
	\P\bigl( \thetavs \not\in \CS(\ZZbt) \bigr)
	&=&
	\P\bigl( L(\tilde{\thetav}) - L(\thetavs) > \ZZbt \bigr)
	\approx 
	\alpha ;
	\label{PtsniCSZbLLZZb}
\end{EQA}
see e.g. \cite{SpZh2014}.
A proof of this result is quite involved. 
The key steps are the following two approximations: 
\begin{EQA}
	\label{suptiTLLtts12}
	\sup_{\thetav \in \Theta} L({\thetav}) - L(\thetavs)
	& \approx &
	\frac{1}{2} \bigl\| \xiv + \av \bigr\|^{2},
	\\
	\sup_{\thetav \in \Theta} \Lb(\thetav) - \Lb(\tilde{\thetav}) 
	& \approx &
	\frac{1}{2} \bigl\| \xivb \bigr\|^{2},
\end{EQA} 
where \( \xiv \) is a Gaussian vector with the variance 
\( \Sigma \) given by
\begin{EQA}
	\Sigma 
	& \eqdef & 
	\DP^{-1} \Var \bigl[ \nabla L(\thetavs) \bigr] \DP^{-1},
	\qquad
	\DP^{2} = - \nabla^{2} \E L(\thetavs) ,
	\label{SigdefDm1DP2}
\end{EQA} 
while \( \xivb \) is conditionally (given \( \Yv \)) Gaussian w.r.t. the bootstrap measure \( \Pb \)
with the covariance \( \Sigmab \) given by
\begin{EQA}
	\Sigmab
	& \eqdef &
	\DP^{-1} \left( \sumi \nabla \ell_{i}(Y_{i},\thetav) \bigl\{ \nabla \ell_{i}(Y_{i},\thetav) \bigr\}^{\T} 
	\right) \DP^{-1} .
	\label{SbdefDm1s1nDm1}
\end{EQA}
The vector \( \av \) in \eqref{suptiTLLtts12} is the so called modeling bias and it vanishes if 
the parametric assumption \( \P = \P_{\thetavs} \) is precisely fulfilled. 
The matrix Bernstein inequality ensures that \( \Sigmab \) is close to \( \Sigma \) in the operator norm
for \( n \) large; see e.g. \cite{Tropp2012}.
This yields bootstrap validity under the true parametric assumption in a weak sense.
However, for quantifying the quality of the bootstrap approximation one has to measure the distance 
between two high dimensional Gaussian distributions \( \ND(\av,\Sigma) \) and \( \ND(0,\Sigmab) \).
The recent paper \cite{SpZh2014} used the approach based on the Pinsker inequality which gives a bound in the total variation 
distance \(\| \cdot\|_{\TV} \) via the Kullback-Leibler divergence between these two measures.
A related bound involves the Frobenius norm  \(\| \cdot\|_{\Fr} \)  of the matrix 
\( \Sigma^{-1/2} \Sigmab \Sigma^{-1/2} - \Id_{\dimp} \)
and the norm of the vector \( \betav \eqdef \Sigma^{-1/2} \av \):
\begin{EQA}
	\bigl\| \ND(\av,\Sigma) - \ND(0,\Sigmab) \bigr\|_{\TV}
	& \leq &
	\frac{1}{2} \Bigl( 
	\bigl\| \Sigma^{-1/2} \Sigmab \Sigma^{-1/2} - \Id_{\dimp} 
	\bigr\|_{\Fr}
	+ \bigl\| \Sigma^{-1/2} \av \bigr\| \Bigr);
	\qquad
	\label{TVdistfrb1}
\end{EQA}
see e.g. \cite{SpZh2014}.
However, if we limit ourselves to the centered balls
then these bounds can be significantly improved.
Namely, by the main result of Theorem~\ref {l: explicit gaussian comparison} and Corollary~\ref{Tgaussiancomparison3} below, we get under some technical conditions
\begin{EQA}
	\left| 
	\P\Bigl( \bigl\| \xiv + \av \bigr\|^{2} > 2\ZZbt \Bigr) - \alpha 
	\right|
	& \leq &
	\frac{\CONST}{\| \Sigma\|_{\Fr}} \Bigl( 
	\| \Sigma - \Sigmab \|_{1} + \| \av \|^{2}
	\Bigr) .
	\label{PG1QvttcYC12}
\end{EQA}
The ``small modeling bias'' condition on \( \av \)  from \cite{SpZh2014} means that the value 
\( \| \Sigma^{-1/2} \av \| \) is small and 
it ensures that a possible model misspecification does not destroy the validity of the bootstrap. 
Comparison of \eqref{PG1QvttcYC12}  with \eqref{TVdistfrb1} reveals a number of benefits of \eqref{PG1QvttcYC12}.
First, the ``shift'' term is proportional to the squared norm of the vector \( \av \),
while the bound \eqref{TVdistfrb1}  depends on the norm of 
\( \Sigma^{-1/2} \av \), i.e. on the whole spectrum of \( \Sigma \).   
Normalization by \( \Sigma^{-1/2} \) can significantly inflate
the vector \( \av \) in directions where the eigenvalues of 
\( \Sigma \) are small.
In the contrary, the bound \eqref{PG1QvttcYC12} only involves 
the squared norm \( \| \av \|^{2} \) and the Frobenius norm of \( \Sigma \),
and the improvement from \( \bigl\| \Sigma^{-1/2} \av \bigr\| \) 
to \( \| \av \|^{2}/\| \Sigma\|_{\Fr} \) can be enormous if some eigenvalues 
of \( \Sigma \) nearly vanish.
Further, the Frobenius norm 
\( \bigl\| \Sigma^{-1/2} \Sigmab \Sigma^{-1/2} - \Id_{\dimp} \bigr\|_{\Fr} \) can be much larger than the ratio
\( \bigl\| \Sigma - \Sigmab \bigr\|_{1} \big/ \| \Sigma\|_{\Fr} \)
by the same reasons.

\subsubsection{Prior impact in linear Gaussian modeling}
\label{motivation 2}
Consider a linear regression model
\begin{EQA}
	Y_{i}
	&=&
	\Psi_{i}^{\T} \thetav + \varepsilon_{i}
	\label{YiPsiiTtts}
\end{EQA}
The assumption of homogeneous Gaussian errors \( \eps_{i} \sim \ND(0,\sigma^{2}) \) yields the log-likelihood
\begin{EQA}
	L(\thetav)
	&=&
	- \frac{1}{2 \sigma^{2}} \sumi (Y_{i} - \Psi_{i}^{\T} \thetav)^{2}
	+ R
	=
	- \frac{1}{2 \sigma^{2}} \bigl\| \Yv - \Psiv^{\T} \thetav \bigr\|^{2} + R ,
	\label{LtheR12}
\end{EQA}
where the term \( R \) does not depend on \( \thetav \).
A Gaussian prior \( \prior = \prior_{\GP} = \ND\bigl( 0, \GP^{-2} \bigr) \) results in
the posterior
\begin{EQA}
	\vthetav_{\GP} \cond \Yv
	& \propto &
	\exp\left( L(\thetav) - \frac{1}{2} \| \GP \thetav \|^{2} \right)
	\propto
	\exp\left( 
	- \frac{1}{2 \sigma^{2}} \bigl\| \Yv - \Psiv^{\T} \thetav \bigr\|^{2} 
	- \frac{1}{2} \| \GP \thetav \|^{2} 
	\right) .
	\label{vtYvLG1212}
\end{EQA}
We shall represent the quantity 
\( \LGP(\thetav) \eqdef L(\thetav) - \frac{1}{2} \| \GP \thetav \|^{2} \) in the form
\begin{EQA}
	\LGP(\thetav)
	&=&
	\LGP(\thetavb_{\GP}) - \frac{1}{2} \bigl\| \DPGP (\thetav - \thetavb_{\GP}) \bigr\|^{2} ,
	\label{LGPthtGDtt}
\end{EQA}
where
\begin{EQA}
	\thetavb_{\GP}
	& \eqdef &
	\bigl( \Psiv \Psiv^{\T} + \sigma^{2} \GP^{2} \bigr)^{-1} \Psiv \Yv ,
	\\
	\DPGP^{2}
	& \eqdef &
	\sigma^{-2} \Psiv \Psiv^{\T} + \GP^{2} .
	\label{tGDG22m2PPG}
\end{EQA}
In particular, it implies that the posterior distribution \( \P(\vthetav_{\GP} \cond \Yv) \) 
of \( \vthetav_{\GP} \) given \( \Yv \) is \( \ND(\thetavb_{\GP},\DPGP^{-2}) \).
A contraction property is a kind of concentration of the posterior on the elliptic set
\begin{EQA}
	E_{\GP}(\rr)
	&=&
	\bigl\{ \thetav \colon \| \QL (\thetav - \thetavb_{\GP}) \| \leq \rr \bigr\} ,
	\label{ErtWttGr}
\end{EQA}
where \( \QL \) is a given linear mapping from \( \R^{\dimp} \).
The desirable credibility property manifests the prescribed conditional probability 
of \( \vthetav_{\GP} \in E(\rr_{\GP}) \) given \( \Yv \) 
with \( \rr_{\GP} \) defined for a given \( \alpha \) by
\begin{EQA}
	\P\Bigl( \bigl\| \QL \bigl( \vthetav_{\GP} - \thetavb_{\GP} \bigr) \bigr\| 
	\geq \rr_{\GP} \cond \Yv \Bigr)
	& = &
	\alpha .
	\label{PGYvtpiG}
\end{EQA}
Under the posterior measure \( \vthetav_{\GP} \sim \ND(\thetavb_{\GP},\DPGP^{-2}) \), this bound
reads as
\begin{EQA}
	\P\bigl( \| \xiv_{\GP} \| \geq \rr_{\GP} \bigr)
	&=&
	\alpha \, 
	\label{PWDGm1xir}
\end{EQA}
with a zero mean normal vector \( \xiv_{\GP} \sim \ND(0, \Sigma_{\GP}) \)
for \( \Sigma_{\GP} = \QL \DPGP^{-2} \QL^{\T} \).
The question of a prior impact can be stated as follows:
whether the obtained credible set significantly depends on the prior covariance \( \GP \). 
Consider another prior \( \prior_{1} = \ND(0,\GP_{1}^{-2}) \) with the covariance matrix
\( \GP_{1}^{-2} \).
The corresponding posterior \( \vthetav_{\GP_{1}} \) is again normal but now with parameters
\( \thetavb_{\GP_{1}} 
= \bigl( \Psiv \Psiv^{\T} + \sigma^{2} \GP_{1}^{2} \bigr)^{-1} \Psiv \Yv \) and 
\( \DP_{\GP_{1}}^{2} = \sigma^{-2} \Psiv \Psiv^{\T} + \GP_{1}^{2} \).
We aim at checking the posterior probability of the credible set \( E_{\GP}(\rr_{\GP}) \): 
\begin{EQA}[c]
	\P\Bigl( \bigl\| \QL \bigl( \vthetav_{\GP_{1}} - \thetavb_{\GP} \bigr) \bigr\| 
	\geq \rr_{\GP} \cond \Yv \Bigr).
	\label{PG1QvttGrG}
\end{EQA}
Clearly this probability can be written as
\begin{EQA}[c]
	\P\Bigl( \bigl\| \xiv_{\GP_{1}} + \av \bigr\| \geq \rr_{\GP} \Bigr)
	\label{PxiG1marG1}
\end{EQA}
with \( \xiv_{\GP_{1}} \sim \ND(0,\Sigma_{\GP_{1}}) \) for 
\( \Sigma_{\GP_{1}} = \QL \DP_{\GP_{1}}^{-2} \QL^{\T} \) and 
\begin{EQA}
	\av
	& \eqdef &
	\QL\bigl( \thetavb_{\GP_{1}} - \thetavb_{\GP} \bigr) .
	\label{aWtGtG1def}
\end{EQA}
Therefore, 
\begin{EQA}
	\left| 
	\P\Bigl( 
	\bigl\| \QL \bigl( \vthetav_{\GP_{1}} - \thetavb_{\GP} \bigr) \bigr\|	\geq \rr_{\GP} \cond \Yv 
	\Bigr) - \alpha 
	\right|
	& \leq &
	\sup_{\rr > 0} \left| 
	\P\Bigl( \bigl\| \xiv_{\GP_{1}} - \av \bigr\| \geq \rr \Bigr) 
	- \P\Bigl( \bigl\| \xiv_{\GP} \bigr\| \geq \rr \Bigr)
	\right| \, .
	\label{PG1WvttGrGPr}
\end{EQA}
Again, the Pinsker inequality allows to upperbound the total variation distance between 
the Gaussian measures \( \ND(0,\Sigma_{\GP}) \) and 
\( \ND(\av,\Sigma_{\GP_{1}}) \), 
however the answer is given via the Kullback-Leibler distance
between these two measures:
\begin{EQA}
	\bigl\| \ND(0,\Sigma_{\GP}) - \ND(\av,\Sigma_{\GP_{1}}) \bigr\|_{\TV}
	& \leq &
	\CONST \Bigl( 
	\bigl\| \Sigma_{\GP}^{-1/2} \Sigma_{\GP_{1}} \Sigma_{\GP}^{-1/2} - \Id_{\dimp} 
	\bigr\|_{\Fr}
	+ \bigl\| \Sigma_{\GP_{1}}^{-1/2} \av \bigr\| \Bigr);
	\qquad
	\label{TVdistGG1}
\end{EQA}
see e.g. \cite{PaSp2013}.
Results of this paper allow to significantly improve this bound. 
In particular, only the nuclear norm 
\( \bigl\| \Sigma_{\GP} - \Sigma_{\GP_{1}} \bigr\|_{1} \), the norm of the vector
\( \av \) and the Frobenius norm of \( \Sigma_{\GP} \) are involved.
If \( \GP^{2} \geq \GP_{1}^{2} \), then \( \Sigma_{\GP} \leq \Sigma_{\GP_{1}} \) and
\begin{EQA}
	\bigl\| \Sigma_{\GP} - \Sigma_{\GP_{1}} \bigr\|_{1}
	&=&
	\tr \Sigma_{\GP_{1}} - \tr \Sigma_{\GP}
	\label{SGmSG1tGtG1}
\end{EQA}
and thus, by the main result of Theorem~\ref {l: explicit gaussian comparison} and Corollary~\ref{Tgaussiancomparison3} below, it holds under some technical conditions
\begin{EQA}
	\left| 
	\P\Bigl( \bigl\| \QL \bigl( \vthetav_{\GP_{1}} - \thetavb_{\GP} \bigr) \bigr\| 
	\geq 
	\rr_{\GP} \cond \Yv \Bigr) - \alpha 
	\right|
	& \leq &
	\frac{\CONST\bigl( \tr \Sigma_{\GP_{1}} - \tr \Sigma_{\GP} + \| \av \|^{2} \bigr)}
	{\| \Sigma_{\GP}\|_{\Fr}} \, .
	\label{PG1QvttGrGcYC12}
\end{EQA}
This new bound significantly outperforms \eqref{TVdistGG1};
see the discussion at the end of Section~\ref{motivation 1}.

\subsubsection{Nonparametric Bayes approach}
\label{motivation 3}

One of the central question in the \emph{nonparametric Bayes} approach 
is whether one can use the corresponding credible set as 
a \emph{frequentist confidence set}
for the true underlying mean \( \E \Yv = \fvs = \Psiv^{\T} \thetavs \).
Here we consider the model 
\( \Yv = \fvs + \epsv = \Psiv^{\T} \thetav + \epsv \) in \( \R^{n} \) with a homogeneous Gaussian noise 
\( \epsv \sim \ND(0,\sigma^{2} \Id_{n}) \) and a Gaussian prior \( \ND(0,\GP^{-2}) \) on \( \thetav \).
The credible set \( E_{\GP}(\rr) \) for \( \vthetav_{\GP} \) yields the credible set \( \CS_{\GP}(\rr) \) for the corresponding response  
\( \fv = \Psiv^{\T} \thetav \):
\begin{EQA}
	\CS(\rr)
	&=&
	\bigl\{ \fv = \Psiv^{\T} \thetav \colon \| \WF \, \Psiv^{\T} (\thetav - \thetavb_{\GP}) \| 
	\leq \rr \bigr\} ,
	\label{ErtWttGrNB}
\end{EQA}
with some linear mapping \( \WF \).
The radius \( \rr = \rr_{\GP} \) is fixed to ensure the prescribed credibility \( 1 - \alp \) for 
the corresponding set \( \CS(\rr_{\alp}) \)
due to \eqref{PGYvtpiG} or \eqref{PWDGm1xir} with \( \QL = \WF \Psiv^{\T} \) and 
\( \Sigma_{\GP} = \WF \Psiv^{\T} \DPGP^{-2} \Psiv \WF^{\T} = \sigma^{2} \WF \Pi_{\GP} \WF^{\T} \), with \( \Pi_{\GP} = \Psiv^{\T} \bigl( \Psiv \Psiv^{\T} + \sigma^{2} \GP^{2} \bigr)^{-1} \Psiv \).
The frequentist coverage probability of the true response \( \fvs \) is given by 
\begin{EQA}
	\P\bigl( \fvs \in \CS_{\GP}(\rr) \bigr)
	&=&
	\P\bigl( \| \WF (\fvs - \Psiv^{\T} \thetavb_{\GP}) \| \leq \rr \bigr) 
	=
	\P\bigl( \| \WF \, \Psiv^{\T} (\thetavs -\thetavb_{\GP}) \| \leq \rr \bigr) .
	\label{PfvsniErCfPG}
\end{EQA}
The aim is to show that the the latter is close to \( 1 - \alp \).
For the posterior mean 
\( \thetavb_{\GP} = \bigl( \Psiv \Psiv^{\T} + \sigma^{2} \GP^{2} \bigr)^{-1} \Psiv \Yv \),
it holds 
\begin{EQA}
	\E \bigl[ \WF \bigl( \fvs - \Psiv^{\T} \thetavb_{\GP} \bigr) \bigr] 
	&=& 
	\WF \bigl( \Id - \Pi_{\GP} \bigr) \fvs 
	\eqdef
	\av.
	\label{adefCPiGfs}
\end{EQA}
Further,
\begin{EQA}
	\Sigmafr
	\eqdef
	\Var\bigl\{ \WF \bigl( \fvs - \Psiv^{\T} \thetavb_{\GP} \bigr) \bigr\}
	&=&
	\Var \bigl\{ \WF \Pi_{\GP} \, \epsv \bigr\}
	=
	\sigma^{2} \WF \Pi_{\GP}^{2} \WF^{\T}
	\label{SiFrCPiG2CT}
\end{EQA}
and hence, the vector \( \WF \bigl( \fvs - \Psiv^{\T} \thetavb_{\GP} \bigr) \) is under 
\( \P \) normal with mean \( \av = \WF \bigl( \Id - \Pi_{\GP} \bigr) \fvs \) and 
variance \( \Sigmafr = \sigma^{2} \WF \Pi_{\GP}^{2} \WF^{\T} \).
Therefore, 
\begin{EQA}
	\P\bigl( \fvs\in \CS_{\GP}(\rr) \bigr)
	&=&
	\P\bigl( \bigl\| \av + \xivfr \bigr\| \leq \rr \bigr). 
	\label{PfsniErPapB}
\end{EQA}
Here \( \xivfr \sim \ND(0, \Sigmafr) \).
So, it suffices to compare two probabilities
\begin{EQA}
	\P\bigl( \bigl\| \av + \xivfr \bigr\| \leq \rr \bigr)
	& \text{ vs } &
	\P\bigl( \bigl\| \xiv_{\GP} \bigr\| \leq \rr \bigr)
	\label{PaSf12xiPS12r}
\end{EQA}
for all \( \rr \geq 0 \).
Existing results cover only very special cases; see e.g. 
\cite{Jo2012, Bo2011, PaSp2013, Ca2012, CaNi2013, belitser2017} and references therein.
Most of the mentioned results are of asymptotic nature and do not quantify the accuracy of the coverage probability. 
The results of this paper enable to study this accuracy in a straightforward way.
Note first that the covariance operators 
\( \Sigmafr = \sigma^{2} \WF \Pi_{\GP}^{2} \WF^{\T} \)
and 
\( \Sigma_{\GP} = \sigma^{2} \WF \Pi_{\GP} \WF^{\T} \) satisfy \( \Sigmafr \le \Sigma_{\GP} \).
This yields that
\begin{EQA}
	\bigl\| \Sigma_{\GP} - \Sigmafr \bigr\|_{1}
	&=&
	\tr \Sigma_{\GP} - \tr \Sigmafr \, .
	\label{SGSftrSGSf}
\end{EQA}
Theorem~\ref {l: explicit gaussian comparison} and Corollary~\ref{Tgaussiancomparison3} allow to evaluate under some technical conditions the coverage probability of the credibility set 
\begin{EQA}
	\bigl| \P\bigl( \fvs \not\in \CS_{\GP}(\rr_{\GP}) \bigr) - \alpha \bigr|
	& \leq &
	\frac{\CONST\bigl( \tr \Sigma_{\GP} - \tr \Sigmafr + \| \av \|^{2} \bigr)}
	{\| \Sigmafr \|_{\Fr}} \, .
	\label{PfsniEGrGal}
\end{EQA}
The right hand-side of this bound can be easily evaluated.
The value \( \| \av \| = \WF \bigl( \Id - \Pi_{\GP} \bigr) \fvs \) is small 
under usual smoothness assumptions on \( \fvs \).
The difference 
\begin{EQA}
	\tr \Sigma_{\GP} - \tr \Sigmafr
	& = &
	\sigma^{2} \tr \bigl\{ \WF (\Pi_{\GP} - \Pi_{\GP}^{2}) \WF^{\T} \bigr\}
	\label{trSGtrSfs2APi}
\end{EQA} 
is small under standard condition on the design \( \Psiv \) and on the spectrum of \( \GP^{2} \); see e.g. 
\cite{SP2013_rough}.

\subsubsection{Central Limit Theorem in finite- and infinite-dimensional spaces} 
\label{motivation 4}
Another motivation for the current paper comes from the limit theorem in high-dimensional spaces for convex sets, in particular, for non-centred balls. 
Applications of smoothing inequalities require to evaluate the probability of hitting the vicinity of a convex set, see e.g. \cite{Berntkus2003b},~\cite{Bentkus2014}. 
This question is closely related to the anti-concentration inequalities considered below in   Theorem~\ref{band of GE}. 
Recently, significant interest was shown in understanding of the anti-concentration phenomenon for weighted sums of random variables, particularly, in random matrix and number theory. 
We refer the interested reader to~\cite{RudVesh2008},~\cite{Zaitsev2016}.

Let \(Y_{1},\ldots,Y_{n}\) be i.i.d. random vectors in \( \R^{p} \).
Assume that all these vectors have zero mean and the covariance operator \(\Sigma\). Let \(X\) be a Gaussian random vector in \( \R^{p} \)
with zero mean and the same covariance operator \(\Sigma\). 
We are interested to bound 
\begin{EQA}[c]
	\label{Bentkus_1}
	\delta(\Cclass) 
	= 
	\sup_{A \in \Cclass} 
	\left|
	\P\left(\frac{Y_{1}+\dots+Y_{n}}{\sqrt{n}}\in A\right) - \P(X\in A)
	\right|
\end{EQA}
for some class \(\Cclass\) of Borel sets. 
It is worth emphasizing that the probabilities of hitting the vicinities of a set \(A \in \Cclass\), play the crucial role in the form of the bound for \(\delta(\Cclass)\).
Assume the class \(\Cclass\) satisfies the following two conditions:

(i) Class \(\Cclass\) is invariant under affine symmetric transformations, that is, \(\D A + \av \in \Cclass\)
if \(\av \in \R^{p}\) and \(\D: \R^{p}\to \R^{p}\) is a linear symmetric invertible operator.

(ii) Class \(\Cclass\) is invariant under taking \(\varepsilon\)-neighborhoods for all \(\varepsilon > 0\). More precisely,
\( A^{\varepsilon}, A^{-\varepsilon} \in \Cclass \) if \(A \in \Cclass\), where 
\begin{EQA}[c]
	A^{\varepsilon} = \{x \in \R^{p}: \rho_A(x)\leq  \varepsilon \}\,\,\,	\text{and}\,\,\, A^{-\varepsilon} = \{x \in A: B_{\varepsilon}(x) \subset A\},
\end{EQA}
with \(\rho_A(x) = \inf_{y\in A} |x-y| \) as the distance between \(A \subset \R^{p} \) and \(x \in \R^{p}\), and
\( B_{\varepsilon}(x) = \{y \in \R^{p}: |x - y| \leq \varepsilon\}\).

Let \(X_{0}\) be a Gaussian random vector in \( \R^{p} \)
with zero mean and the identity covariance operator \(\Id\). Assume that the class  \(\Cclass\) in \eqref{Bentkus_1} is such that  for all \(A \in \Cclass\) and \(\varepsilon > 0 \) 
\begin{EQA}[c]
	\label{Bentkus_2}
	\P(X_{0} \in A^{\varepsilon} \backslash A)\leq a_p\, \varepsilon,\,\,\,\,\,\, \P(X_{0} \in A \backslash A^{-\varepsilon})\leq a_p\, \varepsilon,
\end{EQA}
where \(a_p = a_p(\Cclass)\) is the so called isoperimetric constant of  \(\Cclass\), e.g. taking \(\Cclass\) as the class of all convex sets in \(\R^{p}\) we get \(a_p \leq 4\,p^{1/4}\); see~\cite{Ball1993}. 

It is known (see~\cite{Bentkus2014}[Theorem 1.2]) that if \(\Cclass\) satisfies conditions (i), (ii) and~\eqref{Bentkus_2} then for some absolute constant \(C\) one has
\begin{EQA}[c]
	\label{Bentkus_3}
	\delta(\Cclass)\leq C\,(1+a_p)\,\E |Y_{1}|^3/\sqrt{n}.
\end{EQA}
Therefore, the inequalities~(\ref{Bentkus_2}), i.e. knowledge of \(a_p\), play the crucial role in the form of the bound~\eqref{Bentkus_3}.

We have a similar situation  in infinite-dimensional spaces. Though contrary to the finite dimensional case even if  \(\Cclass\) is a rather small class of "good" subsets, e.g. the  class of all balls, the convergence of 
\(\P\Bigl( {(Y_{1}+\dots+Y_{n})}/{\sqrt{n}}\in A \Bigr)\) to 
\(\P\bigl( X\in A \bigr)\) for each \(A \in \Cclass\), implied by the central limit theorem, can not be uniform in \(A \in \Cclass\); see e.g. ~\cite{Sazonov1981}[pp. 69--70]. 
However, the convergence becomes uniform for a class of all balls with center at some fixed point, say \(\av\). 
Such classes naturally appear in various statistical problems; see e.g.~\cite{ProkhUlyanov2013} or our previous 
application examples. 
Thus, similar to the inequalities \eqref{Bentkus_2} we need to get sharp bounds for 
the probability \( \P(x < \|X - \av\|^{2} < x + \varepsilon) \)
for the Gaussian element \( X \) in a Hilbert space \( \HM \).
Due to our Theorem~\ref{band of GE} below, it holds under some technical conditions that
\begin{EQA}[c]
	\P\bigl( x < \|X - \av\|^{2} < x + \varepsilon \bigr)
	\leq
	\frac{\CONST \, \varepsilon}{\| \Sigma \|_{\Fr}} \, 
\end{EQA}
for an absolute constant \( \CONST \).

\section{Main results}
\label{SmainresGC}

Throughout the paper the following notation are used. 
We write \( a \lesssim b \) (\( a \gtrsim b \)) if there exists some absolute constant \( C \) such that \( a \le C b \) (\( a \geq C b \) resp.). 
Similarly, \(a \asymp b\) means that there exist \(c, C\) such that \(c \, a \le b \le C \, a \).  
\( \R \) (resp. \( \C \)) denotes the set of all real (resp. complex) numbers.
We assume that all random variables are defined on common probability space \( (\Omega, \mathfrak{F}, \P) \) and take values in a real separable Hilbert space \(\HM \) with a scalar product \( \langle\cdot, \cdot\rangle \)  and  norm \( \| \cdot \| \). If dimension of \(\HM \) is finite and equals \(p \), we shall write \(\R^p \) instead of \(\HM \).  Let \( \E \) be the mathematical expectation with respect to \( \P \).  
We also denote by \(\mathfrak{B}(\HM)\) the Borel \(\sigma\)-algebra. 

For a self-adjoint operator  \(\A \) with eigenvalues \(\lambda_{k}(\A), k \geq 1\), let us denote by \( \| \A \| \) and \( \| \A \|_{1} \) the operator and nuclear (Schatten-one) norm by \( \| \A \| \eqdef \sup_{\|x\|=1} \| \A x\| \) and 
\begin{EQA}[c]
	\| \A \|_{1} \eqdef \tr |\A |= \sum_{k=1}^{\infty} |\lambda_{k}(\A)|. 
\end{EQA}  
We suppose below that \( \A \) is a nuclear and \( \| \A \|_{1} < \infty \).

Let \(\Sigma_{\xiv}\) be a covariance operator  of an arbitrary Gaussian random element in \(\HM\). By \(\{\lambda_{k\xiv}\}_{k \geq 1}\) we denote the set of its eigenvalues arranged in the non-increasing order, i.e. \(\lambda_{1\xiv} \geq \lambda_{1\xiv} \geq \ldots  \), and let 
\( \lambdav_{\xiv} \eqdef \diag(\lambda_{j \xiv})_{j=1}^{\infty} \). 
Note that  \(\sum_{j=1}^{\infty} \lambda_{j \xiv} < \infty\).  
Introduce the following quantities
\begin{EQA}[c]
	\Frobg_{k\xiv}^{2} 
	\eqdef 
	\sum_{j=k}^{\infty} \lambda_{j \xiv}^{2}, \quad k = 1,2,
	\label{Lambda def}
\end{EQA} 
and
\begin{EQA}
	\CONSTdlt(\Sigma_{\xiv})
	& = &
	\begin{cases}
		\Frobg_{1\xiv}^{-1} \, ,
		& \text{if } 
		3  \lambda_{1,\xiv}^{2} \le \Frobg_{1\xiv}^{2} \, ,  
		\\
		(\lambda_{1\xiv}\Frobg_{2\xiv})^{-1/2}  , 
		& \text{if } 3  \lambda_{1\xiv}^{2} > \Frobg_{1\xiv}^{2},\,\,
		3  \lambda_{2\xiv}^{2} \leq \Frobg_{2\xiv}^{2}, 
		\\
		(\lambda_{1\xiv} \lambda_{2\xiv})^{-1/2}  , 
		& \text{if } 3  \lambda_{1\xiv}^{2} > \Frobg_{1\xiv}^{2},\,\,
		3  \lambda_{2\xiv}^{2} > \Frobg_{2\xiv}^{2} .
	\end{cases}
	\label{CdelSxiSeta}
\end{EQA}
It is easy to see that \( \| \Sigma_{\xiv}\|_{\Fr} = \Frobg_{1\xiv} \). Moreover, it is straightforward to check that 
\begin{EQA}[c]
	\frac{0.9}{(\Frobg_{1\xiv}\Frobg_{2\xiv})^{1/2}}
	\leq 
	\CONSTdlt(\Sigma_{\xiv}) 
	\leq  
	\frac{1.8}{(\Frobg_{1\xiv}\Frobg_{2\xiv})^{1/2}}.
\end{EQA}
Hence, \( 	\CONSTdlt(\Sigma_{\xiv})  \asymp (\Frobg_{1\xiv}\Frobg_{2\xiv})^{-1/2} \) and therefore equivalent results can be formulated in terms of any of the quantities introduced.
The following theorem is our main result.

\begin{theorem}
	\label{l: explicit gaussian comparison}
	Let \(\xiv\) and \(\etav\) be Gaussian elements in \(\HM\) with zero mean and covariance operators 
	\(\Sigma_{\xiv}\) and \(\Sigma_{\etav}\) respectively. 
	For any \( \av \in \HM \) 
	\begin{EQA}[rcl]
		&& \nquad
		\sup_{x > 0} \left|\P( \| \xiv - \av \| \leq x) - \P( \| \etav  \| \leq x) \right| 
		\\ 
		&&
		\lesssim  
		\Bigl\{ \CONSTdlt(\Sigma_{\xiv}) + \CONSTdlt(\Sigma_{\etav}) \Bigr\}
		\bigg( \| \lambdav_{\xiv} - \lambdav_{\etav}  \|_{1} + \| \av\|^{2}\bigg).
		\label{expl_gauss}
	\end{EQA}
\end{theorem}
The proof of Theorem~\ref{l: explicit gaussian comparison}
is given in Section~\ref{SproofsGC}. 

We can see that the obtained bounds can be expressed in terms of the specific 
characteristics of the matrices \( \Sigma_{\xiv} \) and \( \Sigma_{\etav} \)
such as their operator and the Frobenius norms rather than the dimension \( p \).
Another nice feature of the obtained bounds is that they do not involve 
the inverse of \( \Sigma_{\xiv} \) or \( \Sigma_{\etav} \).
In other words, small or vanishing eigenvalues of \( \Sigma_{\xiv} \) or
\( \Sigma_{\etav} \) 
do not affect the obtained bounds in the contrary to the Pinsker bound.
Similarly, only the squared norm \( \| \av \|^{2} \) of the shift \( \av \)  
shows up in the results, while the Pinsker bound involves 
\( \| \Sigma_{\xiv}^{-1/2} \av \| \) which can be very large or infinite if 
\( \Sigma_{\xiv} \) is not well conditioned.

The representation \eqref{CdelSxiSeta} mimics well the three typical situations:
in the ``large-dimensional case'' with three or more significant eigenvalues
\( \lambda_{j\xiv} \), one can take 
\( \CONSTdlt(\Sigma_{\xiv}) = \| \Sigma_{\xiv} \|_{\Fr}^{-1} = \lambda_{1\xiv}^{-1} \).
In the ``two dimensional'' case, when the sum \( \Frobg_{k\xiv}^{2} \) is 
of the order \( \lambda_{k\xiv}^{2} \) for \( k=1,2 \), 
the bound only depends on the product \( (\lambda_{1\xiv} \lambda_{2\xiv})^{-1/2} \).
The intermediate case of a spike model with one large eigenvalue 
\( \lambda_{1\xiv} \) and many small eigenvalues \( \lambda_{j\xiv}, j \geq 2 \),
the bound depends on \( (\lambda_{1\xiv} \Frobg_{2\xiv})^{-1/2} \).
%

As it was mentioned earlier, the result of Theorem~\ref{expl_gauss} may be equivalently formulated in a ``unified'' way in terms of  \( (\Frobg_{1\xiv}\Frobg_{2\xiv})^{-1/2} \). 
Moreover, we specify the bound \eqref{expl_gauss}  
in the ``high-dimensional'' case which means at least three significantly 
positive eigenvalues of the matrices \( \Sigma_{\xiv} \) and \( \Sigma_{\etav} \).

\begin{corollary}
	\label{Tgaussiancomparison3}
	Let \(\xiv\) and \(\etav\) be Gaussian elements in \(\HM\) with zero mean and covariance operators 
	\(\Sigma_{\xiv}\) and \(\Sigma_{\etav}\) respectively.
	Then for any \( \av \in \HM \)
	\begin{EQA}[rcl]
		&& \nquad
		\sup_{x > 0} \left|\P( \| \xiv - \av \| \leq x) - \P( \| \etav  \| \leq x) \right| 
		\\ 
		&&
		\lesssim  
		\bigg(\frac{1}{(\Frobg_{1\xiv}\Frobg_{2\xiv})^{1/2}} + \frac{1}{(\Frobg_{1\etav}\Frobg_{2\etav})^{1/2}}\bigg) 
		\bigg( \| \lambdav_{\xiv} - \lambdav_{\etav}  \|_{1} + \| \av\|^{2}\bigg).
		\label{expl_gauss22}
	\end{EQA}
	Moreover, assume that
	\begin{EQA}[c]
		3 \| \Sigma_{\xiv}\|^{2} \le \| \Sigma_{\xiv}\|_{\Fr}^{2} 
		\quad \text{ and } 	\quad 
		3 \| \Sigma_{\etav}\|^{2} \le \| \Sigma_{\etav}\|_{\Fr}^{2} \, .
		\label{2plusdelta}
	\end{EQA} 
	Then for any \( \av \in \HM \) 
	\begin{EQA}
		&&
		\sup_{x > 0} \left|\P( \| \xiv - \av \| \leq x) - \P( \| \etav  \| \leq x) \right| 
		\\
		&& \qquad \lesssim  
		\bigg(
		\frac{1}{\| \Sigma_{\xiv}\|_{\Fr}} 
		+ \frac{1}{\| \Sigma_{\etav}\|_{\Fr}}
		\bigg) 
		\bigg( \| \lambdav_{\xiv} - \lambdav_{\etav} \|_{1} + \| \av\|^{2}\bigg).
		\label{expl_gauss 2}
	\end{EQA}
\end{corollary}

We complement the result of Theorem~\ref{l: explicit gaussian comparison} and Corollary~\ref{Tgaussiancomparison3} with several additional remarks. The first remark is that  by the Weilandt--Hoffman inequality, 
\( \| \lambdav_{\xiv} - \lambdav_{\etav} \|_{1} \le \| \Sigma_{\xiv} - \Sigma_{\etav}\|_{1} \), see e.g.~\cite{MarkusEng}.
This yields the bound in terms of the nuclear norm of 
the difference \( \Sigma_{\xiv} - \Sigma_{\etav} \), which may be more useful in a number of applications.   

\begin{corollary}\label{cor 0} 
	Under conditions of Theorem~\ref{l: explicit gaussian comparison} we have
	\begin{EQA}
		\sup_{x > 0} 
		\left|\P\bigl( \| \xiv - \av \| \leq x) - \P( \| \etav \| \leq x\bigr) \right| 
		& \lesssim &
		\Bigl\{ \CONSTdlt(\Sigma_{\xiv}) + \CONSTdlt(\Sigma_{\etav}) \Bigr\}
		\Bigl( \| \Sigma_{\xiv} - \Sigma_{\etav}  \|_{1} + \| \av\|^{2}\Bigr).
	\end{EQA}
\end{corollary}

The right-hand-side of~~\eqref{expl_gauss} does not change if we exchange \( \xiv \) and \( \etav \) in Theorem~\ref{l: explicit gaussian comparison} and its Corollaries hold for the balls with the same shift \( \av \). 
In particular, the following corollary is true.

\begin{corollary}\label{cor 10}
	Under conditions of Theorem~\ref{l: explicit gaussian comparison} we have 
	\begin{EQA}
		\label{expl_gauss_same_a}
		\sup_{x > 0} \Bigl|\P( \| \xiv - \av \| \leq x) - \P( \| \etav - \av \| \leq x) \Bigr| 
		& \lesssim &
		\Bigl\{ \CONSTdlt(\Sigma_{\xiv}) + \CONSTdlt(\Sigma_{\etav}) \Bigr\} 
		\Bigl( \| \lambdav_{\xiv} - \lambdav_{\etav}  \|_{1} + \| \av\|^{2}\Bigr).
	\end{EQA}	
\end{corollary}

The result of Theorem~\ref{l: explicit gaussian comparison} may be also rewritten in terms of the operator norm
\begin{EQA}[c]
	\| \Sigma_{\xiv}^{-1/2} \Sigma_{\etav} \Sigma_{\xiv}^{-1/2} - \Id \|.
\end{EQA}
Indeed, using the inequality \(\| \A \BB\|_{1} \le \| \A \|_{1} \| \BB\| \) we immediately obtain the following corollary.
\begin{corollary}\label{cor 1}
	Under conditions of Theorem~\ref{l: explicit gaussian comparison} we have
	\begin{EQA}
		&&	\sup_{x > 0} 
		\left|\P( \| \xiv - \av \| \leq x) - \P( \| \etav  \| \leq x) \right| 
		\\ 
		&&\qquad\qquad\qquad 
		\lesssim  
		\Bigl\{ \CONSTdlt(\Sigma_{\xiv}) + \CONSTdlt(\Sigma_{\etav}) \Bigr\} 
		\bigg( \tr \bigl( \Sigma_{\xiv} \bigr) \, 
		\| \Sigma_{\xiv}^{-1/2} \Sigma_{\etav} \Sigma_{\xiv}^{-1/2} - \Id \| 
		+ \| \av \|^{2}
		\bigg).
	\end{EQA}
\end{corollary}

We now discuss the origin of the value \( \CONSTdlt(\Sigma_{\xiv}) \) which appears in the main theorem and its corollaries. 
Analysing the proof of Theorem~\ref{l: explicit gaussian comparison} one may find out that it is necessary to get an upper bound for a probability density function (p.d.f.) \(p_{\xiv}(x)\) (resp. \(p_{\etav}(x)\)) of  \(\| \xiv\|^{2}\) (resp. \(\| \etav \|^{2}\)) and the more general p.d.f. \(p_{\xiv}(x, \av)\) of \(\| \xiv - \av\|^{2}\) for all  \(\av\in \HM \). 
The same arguments remain true for \(p_{\etav}(x)\).  The following theorem provides uniform bounds.
\begin{theorem}\label{l: density est 2}
	Let \(\xiv\) be a Gaussian element in \(\HM\) with zero mean and covariance operator \(\Sigma_{\xiv}\). 
	Then it holds for any \(\av\) that
	\begin{EQA}[c]
		\label{density bound 0}
		\sup_{x \geq 0} p_{\xiv}(x, \av) 
		\lesssim  
		\CONSTdlt(\Sigma_{\xiv})
	\end{EQA}
	with \( \CONSTdlt(\Sigma_{\xiv}) \) from \eqref{CdelSxiSeta}.
	In particular, \( \CONSTdlt(\Sigma_{\xiv}) \lesssim (\Frobg_{1\xiv}\Frobg_{2\xiv})^{-1/2} \).
\end{theorem}
The proof of this theorem will be given in  Section~\ref{SproofsGC}. 

Since \( \xiv \myeq \sum_{j=1}^{\infty} \sqrt{\lambda_{j\xiv}} Z_{j} \ee_{j\xiv}\), we obtain that \( \| \xiv\|^{2} \myeq \sum_{j=1}^{\infty} \lambda_{j \xiv} Z_{j}^{2} \). 
Here and in what follows \(\{\ee_{j\xiv}\ \}_{j=1}^{\infty}\) is the orthonormal basis formed by the eigenvectors of \(\Sigma_{\xiv}\) corresponding to \(\{\lambda_{1\xiv}\}_{j=1}^{\infty}\). 
In the case \(\HM = \R^{\dimp}\), \( \av = 0, \Sigma_{\xiv} \asymp \Id\) one has that 
the distribution of \( \| \xiv\|^{2} \) is close to standard \(\chi^{2}\) with 
\( \dimp \) degrees of freedom and 
\begin{EQA}[c]\label{identity case}
	\sup_{x \geq 0} p_{\xiv}(x, 0) \asymp p^{-1/2}.
\end{EQA} 
Hence, the bound~\eqref{density bound 0} gives the right dependence on \(\dimp\) because \(\CONSTdlt(\Sigma_{\xiv}) \asymp p^{-1/2}\).   
However, a lower bound for \(\sup_{x \geq 0} p_{\xiv}(x, \av) \) in the general case is still an open question. 
Another possible extension is a non-uniform upper bound for the p.d.f. of 
\( \| \xiv - \av\|^{2} \). 
In this direction for any \(\lambda > \lambda_{1 \xiv}\) we can prove that 
\begin{EQA}[c]
	p_{\xiv}(x,\av) 
	\leq 
	\frac{ \exp \bigl(-(x^{1/2} - \| \av\|)^{2}/(2\,\lambda)\bigr)}
	{\sqrt{2\lambda_{1\xiv}\lambda_{2\xiv}}} 
	\prod_{j=3}^{\infty}\,(1-\lambda_{j\xiv}/\lambda)^{-1/2}; 
\end{EQA}
see Lemma~\ref{l: density est} and remark after it in the Appendix. It is still an open question whether it is possible to replace the \(\lambda_{k\xiv}\)'s in the denominator by \( \Frobg_{k\xiv} \), 
\( k=1, 2\).

A direct corollary of Theorem~\ref{l: density est 2} is the following theorem which states for a rather general situation a dimension-free anti-concentration inequality for the squared norm of a Gaussian element \( \xiv \).
In the ``high dimensional situation'', this anti-concentration bound 
only involves the Frobenius norm of \(\Sigma_{\xiv}\). 

\begin{theorem}[\( \eps\)-band of the squared norm of a Gaussian element]
	\label{band of GE}
	Let \(\xiv\) be a Gaussian element in \(\HM\) with zero mean and a covariance operator \(\Sigma_{\xiv}\). 
	Then for arbitrary \(\eps > 0\), one has
	\begin{EQA}[c]
		\label{band of Gaussian2}
		\sup_{x  > 0} \P(x < \| \xiv - \av \|^{2} < x + \eps)  
		\lesssim  
		\CONSTdlt(\Sigma_{\xiv}) \, \eps
	\end{EQA}	
	with \( \CONSTdlt(\Sigma_{\xiv}) \) from \eqref{CdelSxiSeta}.
	In particular, \(\CONSTdlt(\Sigma_{\xiv}) \) can be replaced by 
	\( (\Frobg_{1\xiv} \, \Frobg_{2\xiv})^{-1/2} \).
\end{theorem}


We finish this section showing that the structure of estimates in Theorem~\ref{l: explicit gaussian comparison} and Theorem~\ref{band of GE} is the right one. 

For simplicity, we consider the case of centred ball, i.e. \( \av = 0 \) and denote  \(\CONSTdlt(\Sigma_{\xiv},  \Sigma_{\etav}) \eqdef \max\{\CONSTdlt(\Sigma_{\xiv}), \CONSTdlt(\Sigma_{\etav} )\} \). We show in the special case that
\begin{EQA}[c]
	\limsup \bigg(\sup_{x > 0} \frac{\left|\P( \| \xiv\| \leq x) - \P( \| \etav\| \leq x) \right|}{\CONSTdlt(\Sigma_{\xiv},  \Sigma_{\etav})\| \Sigma_{\xiv} -\Sigma_{\etav} \|_{1}}\bigg)  \geq \CONST_1,
	\label{low bound}
\end{EQA}
where \(\CONST_1\) is some absolute positive constant and \(\limsup\) is taken w.r.t. \(\max(\lambda_{2\xiv}, \lambda_{2\etav}) \downarrow 0\). Hence, in general it is impossible to obtain the upper bound in Theorem~\ref{l: explicit gaussian comparison}, such that it doesn't tend to infinity when \(\lambda_{2\xiv}\) (or \( \lambda_{2\etav} \)) tends to zero. To show~\eqref{low bound} we construct the following example.  Let \( \xiv \) be a Gaussian vector in \( \R^{3}\) with zero mean and covariance matrix \( \Sigma_{\xiv} = \diag(\lambda_{1\xiv}, \lambda_{2\xiv}, \lambda_{3\xiv})\).  Similarly, let \( \etav\) be a Gaussian vector with zero mean and covariance matrix \( \Sigma_{\etav} = \diag(\lambda_{1\etav}, \lambda_{2\etav}, \lambda_{3\etav})\).  Then
\begin{EQA}[c]
	\sup_{x > 0} \left|\P( \| \xiv\| \leq x) - \P( \| \etav\| \leq x) \right|  \geq \left|\P( \| \xiv \| \leq \sqrt R) - \P( \| \etav\| \leq \sqrt R) \right|, 
\end{EQA}
for some \(R\) which will be chosen later.  Put 
\begin{EQA}[c]
	\mathcal E_{1} \eqdef \Bigl \{  (x_{1}, x_{2}, x_{3}) \in \R^{3}: \sum_{j=1}^{3}\lambda_{j\xiv} x_{j}^{2} \le R  \Bigr \}, \quad \mathcal E_{2} \eqdef \Bigl \{  (x_{1}, x_{2}, x_{3}) \in \R^{3}: \sum_{j=1}^{3}\lambda_{j\etav} x_{j}^{2} \le R  \Bigr \}.
\end{EQA}
Let us take \( \lambda_{1\xiv} = \lambda_{1\etav},  \lambda_{2\xiv} = \lambda_{2\etav}, \lambda_{3\etav} = \lambda_{3\xiv} (1 + \varepsilon) \) for some \(0< \varepsilon < 1\). This choice gives \( \| \Sigma_{\xiv} - \Sigma_{\etav}\|_{1} = \varepsilon \lambda_{3}\) and\( \CONSTdlt(\Sigma_{\xiv},  \Sigma_{\etav}) \asymp (\lambda_{1\xiv} \lambda_{2\xiv})^{-1/2} \). It is straightforward to check that
\begin{EQA}
	\left|\P( \| \xiv \| \leq R) - \P( \| \etav\| \leq R) \right| &=& \frac{1}{(2 \pi)^{3/2}} \int_{\mathcal E_{1} \setminus \mathcal E_{2}} \exp \Bigl( -\frac{x_{1}^{2} + x_{2}^{2} + x_{3}^{2}}{2} \Bigr) \, d x_{1} \, d x_{2} \, d x_{3} \\
	&\geq& \frac{1}{(2 \pi)^{3/2}} (|\mathcal E_{1}| - |\mathcal E_{2}|) \exp\Bigl[ - \frac{R}{2}  \Bigl( \frac{1}{\lambda_{1\xiv}} + \frac{1}{\lambda_{2\xiv}} + \frac{1}{\lambda_{3\xiv}}  \Bigr) \Bigr],
\end{EQA}
where \(|\mathcal E_{i}| \) is a volume of the ellipsoid \(|\mathcal E_{i}|,\, i=1, 2. \) Applying formula for the volume of an ellipsoid we obtain
\begin{EQA}[c]
	|\mathcal E_{1}| - |\mathcal E_{2}| = \frac{4 \pi R^{3/2} \| \Sigma_{\xiv} -\Sigma_{\etav} \|_{1} }{3 \sqrt{ \lambda_{1\xiv} \lambda_{2\xiv}} \lambda_{3\xiv}^{3/2} \sqrt{1+\varepsilon} (1 + \sqrt{1 + \varepsilon}) } > \frac{\pi \| \Sigma_{\xiv} -\Sigma_{\etav} \|_{1} }{ \sqrt{\lambda_{1\xiv}  \lambda_{2\xiv}}} \left(\frac{ R}{2\lambda_{3\xiv}} \right)^{3/2}. 
\end{EQA}
We take \( R = 2 \lambda_{3\xiv}\). Then
\begin{EQA}[c]
	\left(\frac{R}{2\lambda_{3\xiv}} \right)^{3/2} \exp\Bigl( - \frac{R}{2 \lambda_{3\xiv}}  \Bigr) \geq e^{-1} \geq \frac{1}{3}.
\end{EQA}
Hence,
\begin{EQA}
	\left|\P( \| \xiv \| \leq \sqrt R) - \P( \| \etav\| \leq \sqrt R) \right| &\geq& \frac{\| \Sigma_{\xiv} -\Sigma_{\etav} \|_{1}}{16  \sqrt{\lambda_{1\xiv}\lambda_{2\xiv}}} \exp\Bigl[ -  \Bigl(\frac{ \lambda_{3\xiv}}{\lambda_{1\xiv}}  + \frac{\lambda_{3\xiv}}{\lambda_{2\xiv}}  \Bigr) \Bigr]
	\geq \CONST_1 \frac{\| \Sigma_{\xiv} -\Sigma_{\etav} \|_{1}}{ \sqrt{\lambda_{1\xiv} \lambda_{2\xiv}}},
\end{EQA}
where \( \CONST_1 \eqdef  \exp\bigl( -2  \bigr)/16 \). From the last inequality we may conclude~\eqref{low bound}.

We now turn to the case \(\HM = \R^1\). Here, one may get a two-sided inequality. First, we derive an upper bound. Let \( \xiv \) and \(\etav\) be normal variables with zero mean and variances \( \lambda_{\xiv}\) and \(\lambda_{\etav}\) resp. Without loss of generality we may assume that \( \lambda_{\xiv} < \lambda_{\etav} \). Then
\begin{EQA}
	&& \nquad
	\sup_{x > 0} \left|\P( \| \xiv\| \leq x) - \P( \| \etav\| \leq x) \right|  
	=
	\frac{2}{\sqrt{2\pi}} \sup_{x > 0} 
	\int_{x/\sqrt{\lambda_{\etav}}}^{x/\sqrt{\lambda_{\xiv}}} e^{-y^{2}/2} \, dy 
	\\
	& \leq &  
	\frac{\| \Sigma_{\xiv} - \Sigma_{\etav}\|_{1}}{ \sqrt{ \lambda_{\etav} \lambda_{\xiv}} (\sqrt \lambda_{\xiv} + \sqrt \lambda_{\etav})}  
	\sup_{x > 0}  \left( x \exp \bigl( -x^{2}/(2\lambda_{\etav}) \bigr) \right) 
	\lesssim
	\frac{ \| \Sigma_{\xiv} - \Sigma_{\etav}\|_{1} }{ \lambda_{\xiv}}. 
\end{EQA}
We also have the following lower bound:  
\begin{EQA}
	&& \nquad
	\sup_{x > 0} \left|\P( \| \xiv\| \leq x) - \P( \| \etav\| \leq x) \right|  
	= 
	\frac{2}{\sqrt{2\pi}} \sup_{x > 0} 
	\int_{x/\sqrt{\lambda_{\etav}}}^{x/\sqrt{\lambda_{\xiv}}} e^{-y^{2}/2} \, dy 
	\\
	&\geq&  
	\frac{2\,\| \Sigma_{\xiv} - \Sigma_{\etav}\|_{1}\, x_{0} \exp \bigl( -x_{0}^{2}/(2\lambda_{\xiv}) \bigr) }{\sqrt{2\pi} \sqrt{ \lambda_{\etav} \lambda_{\xiv}} (\sqrt \lambda_{\xiv} + \sqrt \lambda_{\etav})}      	\gtrsim  
	\frac{ \| \Sigma_{\xiv} - \Sigma_{\etav}\|_{1} }{ \lambda_{\etav}}, 
\end{EQA}
where \( x_{0} \eqdef \sqrt{\lambda_{\xiv}}\).

Similar arguments can be applied in the case of Theorem~\ref{band of GE}. 
The right-hand side of \eqref{band of Gaussian2}
essentially depends on the first two eigenvalues of \(\Sigma_{\xiv}\). 
In general, it is impossible to get similar bounds of order \(O( \eps)\) with dependence on \(\lambda_{1\xiv}\) only.  In fact, let \( \HM = \R^2 \)
and  \(\lambda_{1\xiv} = 1\) and  \(\lambda_{2\xiv} = 0\) (i.e. \(\xiv\) has the generate Gaussian distribution). Then for all positive \(\eps \leq \log 2\) one has
\begin{EQA}[c]
	\sup_{x  > 0} \P(x < \| \xiv \|^{2} < x + \eps)  
	\geq 
	\eps^{1/2}/(2\sqrt{\pi}).
\end{EQA}

\section{Proofs of the main results}
\label{SproofsGC}

This section collects the proofs of the main results.

\begin{proof}[Proof of Theorem~\ref{l: density est 2}] 
	Let \(\{\ee_{j}\}_{j=1}^{\infty}\) be an orthonormal basis in \(\HM\) formed by the eigenvectors of \(\Sigma_{\xiv}\) corresponding to eigenvalues \(\{\lambda_{1\xiv}\}_{j=1}^{\infty}\).  	In what follows we omit the index \(\xiv\) from the notation. 
	Put \( a_{j}  \eqdef \langle a, \ee_{j}\rangle \) and \( \xi_{j} \eqdef \langle \xiv, \ee_{j} \rangle \). 
	Then \( \xi_{j} \), \( j \geq 1 \), are independent \( \ND(0, \lambda_{j})\) r.v. 
	Let \( g_{j}(x) \), \( j \geq 1,\) (resp. \(f_{j}(t)\)) be the p.d.f (resp. c.f.) of \( (\xi_{j} - a_{j})^{2} \). 
	Moreover, let \( g(m,x), m \geq 1 \) (resp. \(\bar{g}(m,x), m \geq 1\)) be the p.d.f. of \( \sum_{j=1}^{m} (\xi_{j} - a_{j})^{2} \) (resp. \( \sum_{j=m+1}^{\infty} (\xi_{j} - a_{j})^{2} \)). 
	We also introduce the c.f. \(f(m, t)\) of \(g(m,x)\).  
	As 
	\begin{EQA}[c]
		\label{reduction}
		p(x,\av) 
		\leq 
		\int_{-\infty}^{\infty} g(m,y) \, \bar{g}(m, x-y) \, dy 
		\leq 
		\sup_{x \geq 0} g(m, x),
	\end{EQA}
	we may restrict ourselves to the finite dimensional case only, e.d., \( \HM = \R^{m}\), where \( m \) is some large integer. Hence, in what follows we will assume that \(\xiv\) is a \(m\) dimensional vector. 
	
	We separately consider three cases correspondingto the definition \eqref{CdelSxiSeta} of \( \CONSTdlt(\Sigma_{\xiv}) \): 
	\begin{enumerate}[itemsep=0pt,topsep=0pt,partopsep=0pt,parsep=0pt]
		\item \( 3 \lambda_{1}^{2} \le \Frobg_{1}^{2} \);
		\item \( 3 \lambda_{1}^{2} \geq \Frobg_{1}^{2} \), 
		\( 3\lambda_{2}^{2} \geq \Frobg_{2}^{2} \);
		\item \(  3 \lambda_{1}^{2} \geq \Frobg_{1}^{2} \), 
		\( 3 \lambda_{2}^{2} \le \Frobg_{2}^{2} \).
	\end{enumerate}
	
	We start with the case 1. 
	It is straightforward to check that
	\begin{EQA}[c]
		\label{c.f. with shift}
		|f_{j}(t)| \le \frac{1}{(1+4\lambda_{j}^{2} t^{2})^{1/4}},
		\qquad 
		j = 1, \ldots, m. 
	\end{EQA}
	By the inverse formula 
	\begin{EQA}
		\label{step 3-5}
		p(x,\av) 
		&=& 
		\frac{1}{2\pi} \int_{-\infty}^{+\infty} e^{-itx}
		\prod_{j=1}^{m} {f}_{j}(t)\,dt
		\\
		& \leq &
		\frac{1}{2\pi} \int_{-\infty}^{+\infty} 
		\prod_{j=1}^{m} \bigl| f_{j}(t) \bigr|\,dt
		\leq 
		\frac{1}{2\pi} \int_{-\infty}^{+\infty} 
		\prod_{j=1}^{m} \frac{1}{(1+4\lambda_{j}^{2} t^{2})^{1/4}}\,dt .
	\end{EQA}
	Now Lemma~\ref{Lprodcharf} implies the desired bound.

	The proof in case 2 follows from the Lemma~\ref{l: density est} in Section~\ref{Non-unif_bound}. 
	However, as long as a uniform bound is concerned,
	one can simplify the proof. 
	Indeed, similarly to~\eqref{reduction} one can show that for \( m \geq 2 \)
	\begin{EQA}[c]
		\label{step 1-0}
		g(m,x) \le \sup_{x \geq 0} g(2,x).
	\end{EQA}
	It is straightforward to check that
	\begin{EQA}
		g_{j}(x) 
		&=& 
		\frac{1}{2\sqrt{2\pi x \lambda_{j}}}  
		\Bigl[  \exp\Bigl( - \frac{(x^{1/2} - a_{j})^{2}}{2\lambda_{j}} \Bigr) 
		+ \exp\Bigl( -\frac{(x^{1/2} + a_{j})^{2}}{2\lambda_{j}} \Bigr)
		\Bigr] 
		\le 
		\frac{1}{\sqrt{2\pi x \lambda_{j}}}.
		\qquad
		\label{p.d.f. g j}
	\end{EQA}
	This inequality implies that
	\begin{EQA}
		g(2,x) &=& \int_{0}^{x} g_{1}(x - y) g_{2}(y) \, dy  \le \frac{1}{2 \pi \sqrt{\lambda_{1} \lambda_{2}}}  \int_{0}^{x} (x - y)^{-1/2} y^{-1/2} \, dy =  \frac{1}{2 \sqrt{\lambda_{1} \lambda_{2}}}.
	\end{EQA}
	It remains to use the fact that the r.h.s. of the previous inequality can also be bounded by \( \CONST/ \sqrt{\Frobg_{1} \Frobg_{2}} \).

	Finally we consider the case 3.
	Define \( w_{j} \eqdef {\lambda_{j}^{2}}/\Frobg_{2}^{2} \) for \( j \geq 2 \) and rewrite \(\|\xiv\|^2\) as follows
	\begin{EQA}[c]
		\label{step 3-3}
		\|\xiv\|^2 \myeq \,(\xi_{1} - a_{1})^{2} + \Frobg_{2} \, \eta,
	\end{EQA} 
	where \(\eta \eqdef \sum_{j=2}^{m}\sqrt{w_{j}}\,(Z_{j} - a_{j}')^{2}\), 
	\( a_{j}'\eqdef a_{j}/\sqrt{\lambda_{j}}, Z_j \sim \ND(0,1) \). 
	Let \(p_{\eta} \) be the p.d.f. of random variable \(\eta\). 
	The bound \eqref{p.d.f. g j} implies
	\begin{EQA}[c] 
		\label{step 3-4}
		g(m, x) 
		\le \frac{1}{\sqrt{2\pi \lambda_{1}}}
		\int_{0}^{x/\Frobg_{2}}	\frac{p_{\eta}(z)}{\sqrt{x - \Frobg_{2}\,z}} \, dz 
		\leq 
		\frac{\CONST}{ \sqrt{\lambda_{1} \Frobg_{2}}}\, \sup_{x > 0}
		\int_{0}^{x} \frac{p_{\eta}(z)}{\sqrt{x - z}}\, dz.
	\end{EQA}
	Note that \(p_{\eta}(z)\) is bounded by some absolute constant. 
	Indeed, by the inverse formula 
	\begin{EQA}[c]
		\label{step 3-5}
		p_{\eta}(z) = \frac{1}{2\pi} \int_{-\infty}^{+\infty} e^{-itz}
		\prod_{j=2}^{m}\bar{f}_{j}(t)\,dt,
	\end{EQA}
	where \(\bar{f}_{j}(t)\) is the characteristic function of \(\sqrt{w_j}\,(Z_{j} - a_{j}')^{2}\) for \(j = 2,\dots , m \). 
	Similarly to~\eqref{c.f. with shift} we can bound 
	\(|\bar{f}_{j}(t)| \leq (1+4\,w_{j}\,t^{2} )^{-1/4} \)
	and 
	\begin{EQA}
		p_{\eta}(z)
		& \leq &
		\frac{1}{2\pi} \int_{-\infty}^{+\infty} 
		\prod_{j=2}^{m} \bigl| \bar{f}_{j}(t) \bigr|\,dt
		\leq 
		\frac{1}{2\pi} \int_{-\infty}^{+\infty} 
		\prod_{j=2}^{m} \frac{1}{(1+4\,w_{j}\,t^{2} )^{1/4}}\,dt \, .
		\label{peta12piii2p}
	\end{EQA} 
	In view of 
	\( \sum_{j\geq 2} w_{j} = 1 \), Lemma~\ref{Lprodcharf} implies 
	\begin{EQA}
		\sup_{z} p_{\eta}(z) 
		& \lesssim &
		1 .
		\label{supzpetax1}
	\end{EQA}
	Combining this bound with \eqref{p.d.f. g j} and \eqref{step 3-4} yields 
	the upper bound of order 
	\( \bigl( \lambda_{1} \Frobg_{2} \bigr)^{-1/2} \asymp \bigl( \Frobg_{1} \Frobg_{2} \bigr)^{-1/2} \) in case (3).
	This completes the proof of the theorem.
\end{proof}	

\begin{remark}
	We would like to remark that instead of Lemma~\ref{Lprodcharf} one may also apply an alternative approach from~\cite{Ulyanov1987}[Lemma 5].
\end{remark}

\begin{proof}[Proof of Theorems~\ref{l: explicit gaussian comparison}] We split the proof into two parts. In the first part we study the case \(\av = 0\). The second part is devoted to the case \(\Sigma_{\xiv} = \Sigma_{\etav}\). 
	The final estimate will follow by combining the two obtained estimates and the triangular inequality.

	\subsubsection*{Case I\,: \(\av = 0\).}  Without loss of generality we may assume that \(\Sigma_{\xiv} = \lambdav_{\xiv}, \Sigma_{\etav} = \lambdav_{\etav}\), where \(\lambdav_{\xiv} \eqdef \diag(\lambda_{1\xiv}, \lambda_{2\xiv}, \ldots),\, \lambdav_{\etav} \eqdef \diag(\lambda_{1\etav}, \lambda_{2\etav}, \ldots) \) and \(\lambda_{1\xiv} \geq \lambda_{1\xiv} \geq \ldots  \) and similarly in decreasing order for \(\lambda_{i\etav} \)'s.

	Fix any \(s: 0\leq s \leq 1\).  Let  \(Z(s)\) be a Gaussian random element in \(\HM\) with zero mean and diagonal covariance operator \(\V(s)\):
	\begin{EQA}[c]
		\V(s) \eqdef  s \lambdav_{\xiv} + (1-s) \lambdav_{\etav}.
	\end{EQA}
	Denote by \(f(t,s)\) (resp. \( p(x, s) \)) the characteristic function (resp. p.d.f.) of \(\|Z(s)\|^{2}\).  Let \(\lambda_{1}(s)\geq\lambda_{2}(s)\geq\dots\) be the eigenvalues of \(\V(s)\) and introduce the diagonal resolvent operator  \(\G(t,s) \eqdef (\Id - 2it\, \V(s))^{-1} \). Recall that \( \|Z(s)\|^{2} \myeq \sum_{j=1}^n \lambda_{j}(s) Z_{j}^{2}\), where \(Z_{j}, j \geq 1,\) are i.i.d. \(\ND(0,1)\) r.v. Then it is straightforward to check that a characteristic function \(f(t,s)\) of \(\|Z(s) \|^{2}\) can be written as
	\begin{EQA}
		\label{eq:  c.f.}
		f(t,s) &=& \E \exp\{it\|Z(s)\|^{2}\} \nonumber = \exp \bigg\{ - \frac{1}{2} \tr \log\bigl( \Id - 2 it \V(s) \bigr) \bigg\},
	\end{EQA}
	where for an operator \(\A\) and the identity operator \(\Id\) we use notation  
	\begin{EQA}[c]
		\log(\Id + \A) 
		= 
		\A \int_{0}^{1}(\Id + y \A)^{-1}dy.
	\end{EQA}
	It is well known, see e.g.~\cite{Chung2001}[\S 6.2, p. 168], that for a continues d.f. \(F(x)\) with c.f. \(f(t)\) we may write
	\begin{EQA}[c]
		F(x) 
		= 
		\frac{1}{2} + \frac{i}{2\pi} \lim_{T \rightarrow \infty} \text{V.P.} \int_{|t| \le T} e^{-i t x} f(t) \frac{dt}{t}.
	\end{EQA} 
	Let us fix an arbitrary \(x > 0\). Then 
	\begin{EQA}[c] 
		\P( \| \xiv\|^{2} < x) - \P( \| \etav \|^{2} < x) 
		=  
		\frac{i}{2\pi} \lim_{T \rightarrow \infty} \text{V.P.} 
		\int_{|t| \le T} \frac{f(t,1) - f(t,0)}{t} e^{-i t x} \,dt.
	\end{EQA}
	By the Newton-Leibnitz formula
	\begin{EQA}[c]
		f(t,1) - f(t,0)
		=
		\int_{0}^{1}\frac{\partial f(t,s)}{\partial s}\,ds.
	\end{EQA}
	It is straightforward  to check that
	\begin{EQA}[c]
		\frac{\partial f(t,s)/\partial s}{t} 
		= 
		i f(t,s) \tr\bigl\{ (\lambdav_{\xiv} - \lambdav_{\etav}) \G(t,s) \bigr\} .  
	\end{EQA}
	Changing the order of integration we get
	\begin{EQA}
		\P( \| \xiv \|^{2} < x) &-& \P( \| \etav  \|^{2} < x)
		\nonumber \\
		&=& -\frac{1}{2\pi} \int_{0}^{1}  \int_{-\infty}^{\infty} 
		\tr\left\{ (\lambdav_{\xiv} - \lambdav_{\etav}) \G(t, s) \right\}   
		f(t,s) e^{-itx}\, dt \, ds. 
		\label{eq: difference between measures}
	\end{EQA}
	Since \( \G(t,s) \) is the diagonal operator with \((1 - 2 i t \lambda_{j}(s))^{-1} \) on the  diagonal, we may fix \( s \) and \(j \) and consider the following quantity
	\begin{EQA}[c]
		\frac{1}{2\pi} \int_{-\infty}^{\infty} (1 - 2 i t \lambda_{j}(s))^{-1} f(t,s) e^{-itx} \, dt.
	\end{EQA} 
	Let \( \Zbar_{j}(s), j \geq 1 \) be independent exponentially distributed r.v. with parameter \(1/(2\lambda_{j}(s))\) (we write \(\Exp(2\lambda_{j}(s))\)), which  are also independent of \(Z_{k}, k \geq 1\). 
	Then
	\begin{EQA}[c]
		\label{c.f. Z}
		\E e^{i t \Zbar_{j}(s)} 
		= 
		(1 - 2 i t \lambda_{j}(s))^{-1}.
	\end{EQA}
	Moreover, \((1 - 2 i t \lambda_{j}(s))^{-1} f(t,s)\) is the characteristic function of \( \Zbar_{j}(s) + \| Z(s)\|^{2}\). 
	Let \( p_{j}(x,s) \) be the corresponding p.d.f.  
	Then
	\begin{EQA}[c]
		\frac{1}{2\pi} \int_{-\infty}^{\infty} (1 - 2 i t \lambda_{j}(s))^{-1} f(t,s) e^{-itx} \, dt 
		= 
		p_{j}(x,s).
	\end{EQA}
	Denote by \(  \OP(x,s)\) a diagonal operator with \( p_{j}(x,s) \) on the main diagonal. 
	Then we may conclude that
	\begin{EQA}[c]
		\frac{1}{2\pi} \int_{-\infty}^{\infty} 
		\tr\left\{ (\lambdav_{\xiv} - \lambdav_{\etav}) \G(t,s) \right\} f(t,s) e^{-itx} \, dt 
		= 
		\tr \left\{ (\lambdav_{\xiv} - \lambdav_{\etav})\OP(x,s) \right\}.
	\end{EQA}
	It is clear that the absolute value of the last term is bounded above by
	\begin{EQA}[c]
		\| \lambdav_{\xiv} - \lambdav_{\etav}\|_{1} \,  \max_{j} \sup_{x  \geq 0} p_{j}(x,s)
	\end{EQA}
	and we need to bound uniformly each \( p_{j}(x,s) \). For any \(j\):
	\begin{EQA}[c]
		p_{j}(x,s) 
		= 
		\int_{-\infty}^{\infty} p(y, s) \bar{p}_{j}(x-y,s)\, dy 
		\le 
		\sup_{x \geq 0} p(x,s), 
	\end{EQA}
	where \(\bar{p}_{j}(x,s) \) is the p.d.f. of \(\overline Z_{j}(s)\). 
	Applying Theorem~\ref{l: density est 2} we obtain
	\begin{EQA}[c]
		\sup_{x \geq 0} p(x,s) 
		\lesssim 
		\CONSTdlt(\Sigma(s)) ,
	\end{EQA}
	where \( \CONSTdlt(\Sigma(s)) \) is from \eqref{CdelSxiSeta}.
	It remains to integrate over \(s\) to obtain
	\begin{EQA}[c]
		\sup_{x > 0} \Bigl|\P( \| \xiv\|^{2} < x) - \P( \| \etav \|^{2} < x) \Bigr| 
		\le 
		\Bigl\{ \CONSTdlt(\Sigma_{\xiv}) + \CONSTdlt(\Sigma_{\etav}) \Bigr\}
		\bigl\| \lambdav_{\xiv} - \lambdav_{\etav} \bigr\|_{1}.
	\end{EQA}
	
	\subsubsection*{Case II\,: \(\Sigma_{\xiv} = \Sigma_{\etav}\) and \(\av \ne 0\).} 
	
	We may rotate \(\xiv\) such that \(\Sigma_{\xiv} = \Lambda_{\xiv}\). Then we have to replace \(\av\) by appropriate \( \avc \), but \( \| \av\| = \| \avc\| \). Fix any \(s: 0\leq s \leq 1\). Let \(  \avc(s) \eqdef \avc \sqrt s\). Introduce the diagonal operator \(\G(t) \eqdef (\Id - 2it\, \Lambda_{\xiv})^{-1} \). It is straightforward to check that a characteristic function \(f(t, \avc(s))\) of \(\| \xiv - \avc(s) \|^{2}\) can be written as
	\begin{EQA}
		f(t,\avc(s)) &=& \E \exp\{it\| \xiv - \avc(s) \|^{2}\} \nonumber \\
		&=& \exp \left\{ it\left(s\| \avc \|^{2}  + s \langle \G(t) \avc, \avc \rangle - \frac{1}{2it} \tr \log\bigl( \Id - 2 it \Lambda_{\xiv} \bigr) \right)\right\}.
	\end{EQA}
	Repeating the arguments from the proof of Theorem~\ref{l: explicit gaussian comparison} we obtain  (compare with~\eqref{eq: difference between measures})
	\begin{EQA}
		&& \nquad
		\P( \| \xiv - \av\|^{2} < x) - \P( \| \xiv \|^{2} < x) 
		\\
		&=& 
		-\frac{1}{2\pi}  
		\int_{0}^{1}  \int_{-\infty}^{\infty} \Bigl[\| \avc\|^{2} + \langle\G(t) \avc, \avc \rangle \Bigr] 
		f(t,\avc(s)) e^{-itx}\, dt \, ds.
	\end{EQA} 
	Moreover, we may rewrite the last equation as follows
	\begin{EQA}
		&& \nquad
		\P( \| \xiv - \av\|^{2} < x) - \P( \| \xiv \|^{2} < x) 
		\\
		&=& 
		- \| \av\|^{2} \int_{0}^1 p(x, \avc(s))\, ds 
		- \sum_{j=1}^{\infty} [\ac_{j}]^{2}  \int_{0}^1 p_{j}(x,\avc(s)) \, ds,  
	\end{EQA}
	where \( p(x, \avc(s)),\,  p_{j}(x,\avc(s)) \) are p.d.f of \(\| \xiv - \avc(s)\|^{2}\) and \( \overline Z_{j} + \| \xiv - \avc(s)\|^{2} \) resp. 
	Here \( \overline Z_{j} \) is  a random variable with exponential distribution \(\Exp(2\lambda_{j\xiv})\). 
	It remains to apply Theorem~\ref{l: density est 2} and integrate over \(s\).
\end{proof}

\section{Acknowledgements}

We would like to thank the Associate Editor and the  Reviewer for helpful
comments and suggestions.

F. G{\"o}tze was supported by the German Research Foundation (DFG) through the Collaborative Research Center 1283: ``Taming uncertainty and profiting from randomness and low regularity in analysis, stochastics and their applications''. A.~Naumov was supported RFBR N~16-31-00005 and President's of Russian Federation Grant for young scientists N~4596.2016.1. V. Spokoiny was supported by the Russian Science Foundation (project no. 14 50 00150).
Financial support by the German Research Foundation (DFG) through the Collaborative Research Center 1294 ``Data Assimilation -- The Seamless Integration of Data and Models'' is gratefully acknowledged.

%
%

\appendix
\section{Technical results}

\begin{lemma}
	\label{Lintdelttm1}
	It holds 
	\begin{EQA}
		\sup_{0 < a \leq 1} a	\int_{0}^{\infty}  
		\frac{1}{(1 + \tsqu^{2})^{a+1/2}} d \tsqu
		& \leq &
		\CONST ,
		\label{a012inft2atb12}
	\end{EQA}
	and 
	\begin{EQA}
		\sup_{a \geq 1} a^{1/2} \int_{0}^{\infty}  
		\frac{1}{(1 + \tsqu^{2})^{a+1/2}} d \tsqu
		& \leq &
		\CONST .
		\label{p320inft2atb12}
	\end{EQA}
\end{lemma}

\begin{proof}
	Define
	\begin{EQA}
		\Inta(a)
		& \eqdef &
		\int_{0}^{\infty} \frac{1}{(1 + \tsqu^{2})^{a+1/2}} \, d \tsqu .
		\label{La0inga0ap1}
	\end{EQA}
	Obviously, \( \Inta(a) \) monotonously decreases in \( a \).
	Integration by parts implies for \( a > 0 \)
	\begin{EQA}
		\int_{0}^{\infty} \frac{\tsqu^{2}}{(1 + \tsqu^{2})^{a+3/2}} \, d \tsqu
		&=&
		- \frac{1}{2a+1} 
		\int_{0}^{\infty} \tsqu \, d \left( \frac{1}{(1 + \tsqu^{2})^{a+1/2}} \right) 
		\\
		&=&
		\frac{1}{2a+1} \int_{0}^{\infty} \frac{1}{(1 + \tsqu^{2})^{a+1/2}} \, d \tsqu 
		=
		\frac{\Inta(a)} {2a+1} \, .
		\label{gai12afa2a}
	\end{EQA}
	At the same time, for \( a > 0 \)
	\begin{EQA}
		\int_{0}^{\infty} \frac{\tsqu^{2}}{(1 + \tsqu^{2})^{a+3/2}} \, d \tsqu
		&=&
		\int_{0}^{\infty} \frac{1 + \tsqu^{2}}{(1 + \tsqu^{2})^{a+3/2}} \, d \tsqu
		- \int_{0}^{\infty} \frac{1}{(1 + \tsqu^{2})^{a+3/2}} \, d \tsqu
		=
		\Inta(a) - \Inta(a+1) \, .
		\label{gafafam12a1}
	\end{EQA}
	This implies a recurrent relation
	\begin{EQA}
		\Inta(a+1)
		&=&
		\frac{a}{a+1/2} \Inta(a) \, .
		\label{fa2a2a1fam1}
	\end{EQA}
	For \( a \in [0,1] \), it implies
	\begin{EQA}
		a \Inta(a) 
		&=&
		(a+1/2) \Inta(a+1)
		\leq 
		\frac{3}{2} \, \Inta(1) 
		=
		\CONST 
		\label{aLaa12La1}
	\end{EQA}
	and \eqref{a012inft2atb12} follows. 
	For \( a = a_{0} + k \) with \( a_{0} \in \bigl[ 1,2 \bigr] \) and an integer \( k \geq 0 \), we use that
	\begin{EQA}
		\sqrt{a} \,\, \Inta(a)
		&=&
		\sqrt{a} \, \frac{(a-1)(a-2) \ldots a_{0}}{(a-1/2)(a-3/2) \ldots (a_{0}+1/2)} \Inta(a_{0})
		\\
		&=&
		\frac{\sqrt{a (a-1)}}{a-1/2} \frac{\sqrt{(a-1)(a-2)}}{a-3/2} \ldots 
		\frac{\sqrt{(a_{0}+1) a_{0}}}{a_{0}+1/2} \, \sqrt{a_{0}} \, \Inta(a_{0})
		\leq 
		\sqrt{2} \, \Inta(1)
		=
		\CONST .
		\label{sqaLasaam1am2L1}
	\end{EQA}
	This 
	proves \eqref{p320inft2atb12}.
\end{proof}

\begin{lemma}
	\label{Lprodcharf}
	Let \( \lambda_{1} \geq \lambda_{2} \geq \ldots \geq \lambda_{\dimp} \) and 
	\begin{EQA}
		3 \lambda_{1}^{2} 
		\leq 
		\Frobg^{2}
		& \eqdef &
		\sum_{j=1}^{\dimp} \lambda_{j}^{2} \, .
		\label{L12asj1plj2}
	\end{EQA}
	Define 
	\begin{EQA}
		h_{j}(t)
		& \eqdef &
		\frac{1}{(1 + \lambda_{j}^{2} t^{2})^{1/4}} \,\, ,
		\qquad
		j=1,\ldots,\dimp .
		\label{fjt11lj2t2}
	\end{EQA}
	Then it holds
	\begin{EQA}
		\int_{0}^{\infty} \prod_{j=1}^{\dimp} h_{j}(t) \, dt
		& \lesssim &
		\frac{1}{\Frobg} \, .
		\label{ioipj1fjC1}
	\end{EQA} 
\end{lemma}

\begin{proof}
	Let \( \qq_{j} \) be a set of positive numbers with \( \qq_{j} \geq 3 \) and 
	\( \sum_{j} \qq_{j}^{-1} = 1 \). 
	A specific choice will be given later.
	By the H\"older inequality
	\begin{EQA}
		\int_{0}^{\infty} \prod_{j=1}^{\dimp} h_{j}(t) \, dt
		& \leq &
		\prod_{j=1}^{\dimp} \left( \int_{0}^{\infty} |h_{j}(t)|^{\qq_{j}} \, dt \right)^{1/\qq_{j}} \, .
		\label{i0ipj1pfjpj1pj}
	\end{EQA}
	Further, for each \( j \), by the change of variable \( \lambda_{j} t = u \)
	\begin{EQA}
		\int_{0}^{\infty} |h_{j}(t)|^{\qq_{j}} \, dt
		&=&
		\int_{0}^{\infty} \frac{dt}{(1 + \lambda_{j}^{2} t^{2})^{\qq_{j}/4}}
		=
		\lambda_{j}^{-1} \int_{0}^{\infty} \frac{du}{(1 + u^{2})^{\qq_{j}/4}}
		=
		\lambda_{j}^{-1} \Inta(\qq_{j}/4 - 1/2)
		\label{ljm1Lpj4m1}
	\end{EQA}
	with \( \Inta(\cdot) \) from \eqref{La0inga0ap1}.
	Therefore, by \eqref{p320inft2atb12} of Lemma~\ref{Lintdelttm1} in view of \( \qq_{j}/4 - 1/2 \geq 1/4 \)
	\begin{EQA}
		\int_{0}^{\infty} \prod_{j=1}^{\dimp} h_{j}(t) \, dt
		& \le &
		\prod_{j=1}^{\dimp} \left( \frac{1}{\lambda_{j} \Inta(\qq_{j}/4 - 1/2)} \right)^{1/\qq_{j}}
		\lesssim 
		\prod_{j=1}^{\dimp} \left( \frac{1}{\lambda_{j} \sqrt{\qq_{j}/4 - 1/2}} \right)^{1/\qq_{j}} .
		\qquad
		\label{int0ipj1pj1p1lj}
	\end{EQA}
	Now we fix \( \qq_{j} \) by the condition
	\begin{EQA}
		\lambda_{j}^{2} (\qq_{j}/4 - 1/2)
		&=&
		\tau ,
		\label{C2ljpj4m1}
	\end{EQA}
	where the constant \( \tau \) is determined by \( \sum_{j=1}^{\dimp} \qq_{j}^{-1} = 1 \).
	This yields
	\begin{EQA}[rclcl]
		\frac{1}{\qq_{j}}
		&=&
		\frac{\lambda_{j}^{2}}{4 \tau + 2 \lambda_{j}^{2}} \, ,
		\qquad
		\sum_{j=1}^{\dimp} \frac{\lambda_{j}^{2}}{4 \tau + 2 \lambda_{j}^{2}}
		&=&
		1, 
		\label{1pj1lj24C2lj2}
	\end{EQA}
	and obviously \( \tau \leq \Frobg^{2}/4 \) and 
	\( \tau + \lambda_{1}^{2}/2 \geq \Frobg^{2}/4 \).
	The condition \( 3 \lambda_{1}^{2} \leq \Frobg^{2} \) 
	implies  
	\begin{EQA}
		\qq_{j} 
		& = &
		\frac{4 \tau}{\lambda_{j}^{2}} + 2
		\geq 
		\frac{\Frobg^{2} - 2 \lambda_{1}^{2}}{\lambda_{1}^{2}} + 2
		\geq 3 ,
		\qquad 
		j \leq \dimp.
		\label{pj4C2lj224}
	\end{EQA}
	Also
	\begin{EQA}
		\tau
		& \geq &
		\frac{1}{4} \bigl( \Frobg^{2} - 2 \lambda_{1}^{2} \bigr)
		\geq 
		\frac{1}{4} \Bigl( \Frobg^{2} - \frac{2 \Frobg^{2}}{3} \Bigr)
		\gtrsim
		\Frobg^{2} \, .
		\label{C14F2m2l1dF2}
	\end{EQA}
	Now it follows from \eqref{int0ipj1pj1p1lj} that
	\begin{EQA}
		\int_{0}^{\infty} \prod_{j=1}^{\dimp} h_{j}(t) \, dt
		& \lesssim &
		\biggl( \frac{1}{\sqrt{\tau}} \biggr)^{\qq_{1}^{-1}+ \ldots + \qq_{\dimp}^{-1}}
		\lesssim 
		\frac{1}{\Frobg}
		\label{ipj1pfjCm1}
	\end{EQA}
	as required.
\end{proof}

\section{A non-uniform bound for the density of a weighted non-central  \( \chi^{2}\) distribution} 
\label{Non-unif_bound}

\begin{lemma}\label{l: density est}
	Let \(\xiv\) be a Gaussian element in \(\HM\) with zero mean and covariance operator \(\Sigma_{\xiv}\). For any \(\av\in \HM \) and all \(\lambda>\lambda_{1\xiv}\)
	\begin{EQA}[c]
		\label{eq: density est}
		p_{\xiv}(x,\av) \leq \frac{ \exp \bigl(-(x^{1/2} - \| \av\|)^{2}/(2\,\lambda)\bigr)}{\sqrt{2\lambda_{1\xiv}\lambda_{2\xiv}}} \prod_{j=3}^{\infty}\,(1-\lambda_{j\xiv}/\lambda)^{-1/2}. 
	\end{EQA}
\end{lemma}
\begin{remark}
	The infinite product in the r.h.s. of~\eqref{eq: density est} is convergent. Indeed, taking logarithm and using \(\log (1+x) \geq x/(x+1)\) for \(x >-1\) we obtain
	\begin{EQA}[c]
		0 < -\frac{1}{2}\log \prod_{j=3}^{\infty}\,(1-\lambda_{j\xiv}/\lambda) \le \frac{1}{2(\lambda-\lambda_{1\xiv})}\sum_{j=3}^\infty \lambda_{j\xiv} < \infty,
	\end{EQA} 
	where we also used the fact that \(\Sigma_{\xiv}\) is a nuclear and \( \| \Sigma_{\xiv} \|_{1} < \infty \). Taking \(\lambda = \|\Sigma_{\xiv}\|_1\) we get \(\prod_{j=3}^{\infty}\,(1-\lambda_{j\xiv}/\lambda)^{-1/2} \le \sqrt{e}.\)
\end{remark}
\begin{proof}
	We will use the notation from the proof of Theorem~\ref{l: density est 2}. We rewrite \(g_{j}(x)\) as follows
	\begin{EQA}[c]
		g_{j}(x) = \frac{1}{\sqrt{2\pi x \lambda_{j}}} d_{j}(x),
	\end{EQA}
	where
	\begin{EQA}[c]
		d_{j}(x) \eqdef d_j(\lambda_{j}, x) \eqdef \frac{1}{2} \Bigl[  \exp\bigl(-(x^{1/2} - a_{j})^{2}/(2\lambda_{j})\bigr) + \exp\bigl(-(x^{1/2} + a_{j})^{2}/(2\lambda_{j})\bigr) \Bigr].
	\end{EQA}
	It is straightforward to check that for \(a \geq b \geq 0\)
	\begin{EQA}[c]
		\label{inequality 1}
		((a - b)^{1/2} - c)^{2} + (b^{1/2} - d)^{2}  \geq (a^{1/2} - (c^{2} + d^{2})^{1/2})^{2},
	\end{EQA}
	and
	\begin{EQA}[c]
		\label{dj bound}
		d_{j}(x) \le \exp\bigl(-(x^{1/2} - |a_{j}|)^{2}/(2\lambda_{j})\bigr).
	\end{EQA}
	We have for all \(j=1,2,\dots\) and any \(\lambda > \lambda_{1}\) 
	\begin{EQA}[c]
		\label{eq: g_{j} bound}
		g_{j}(x)
		\le \frac{1}{\sqrt{2\pi x \lambda_{j}}} \exp\bigl(-(x^{1/2} - |a_{j}|)^{2}/(2\lambda)\bigr) d_j(\lambda\lambda_{j}/(\lambda - \lambda_{j}), x).
	\end{EQA}
	Moreover,
	\begin{EQA}[c]
		\label{eq: density}
		(2\pi x)^{-1/2}(\lambda - \lambda_{j})^{1/2}/(\lambda\lambda_{j})^{1/2}\,d_j(\lambda\lambda_{j}/(\lambda - \lambda_{j}), x)
	\end{EQA}
	is the density function of  \( \big(\sqrt{\lambda/(\lambda-\lambda_{j})}\,\xi_{j} - a_{j}\big)^{2} \). These inequalities imply
	\begin{EQA}
		g(2,x) &=& \int_{0}^{x} g_{1}(x - y) g_{2}(y) \, dy \\ &\le& \frac{1}{2 \pi \sqrt{\lambda_{1} \lambda_{2}}} \exp \bigl( -(x^{1/2} - (a_{1}^{2} + a_{2}^{2})^{1/2})^{2} /(2 \lambda) \bigr) \int_{0}^{x} (x - y)^{-1/2} y^{-1/2} \, dy \\
		& = & \frac{1}{2 \sqrt{\lambda_{1} \lambda_{2}}} \exp \bigl( -(x^{1/2} - (a_{1}^{2} + a_{2}^{2})^{1/2})^{2} /(2 \lambda) \bigr).
	\end{EQA}
	Similarly, applying the last inequality,~\eqref{eq: g_{j} bound} and~\eqref{eq: density} we obtain
	\begin{EQA}
		g(3,x) &=& \int_{0}^{x} g(2,x - y) g_{3}(y) \, dy \\ &\le& \frac{1}{2 \sqrt{\lambda_{1} \lambda_{2}} \sqrt{2 \pi \lambda_{3} }} \exp \bigl( -(x^{1/2} - (a_{1}^{2} + a_{2}^{2} + a_{3}^{2})^{1/2})^{2} /(2 \lambda) \bigr) \\
		&\times& \int_{0}^{x} \frac{d_j(\lambda\lambda_{3}/(\lambda - \lambda_{3}), y)}{y^{1/2}} \, dy \\
		& \le & \frac{1}{2 \sqrt{\lambda_{1} \lambda_{2}}} \exp \bigl( -(x^{1/2} - (a_{1}^{2} + a_{2}^{2} + a_{3}^{2})^{1/2})^{2} /(2 \lambda) \bigr)  \left(1 - \frac{\lambda_{3}}{\lambda}\right)^{-1/2}.
	\end{EQA}
	By induction we get
	\begin{EQA}[c]
		\label{density for fixed m}
		g(m,x) \le \frac{1}{2 \sqrt{\lambda_{1} \lambda_{2}}} \exp \Bigl( -\frac{(x^{1/2} - (a_{1}^{2} + \ldots + a_m^{2})^{1/2})^{2}}{2 \lambda} \Bigr) \prod_{j=3}^{m} \left(1 - \frac{\lambda_{j}}{\lambda}\right)^{-1/2}.
	\end{EQA}
	Now take an arbitrary \(\varepsilon > 0\) and any integer \(m>0\). Let \(0 < \mu < 1/(2\lambda_{j})\) for all \(j \geq m+1\). Without loss of generality we assume that at least two \(\lambda_{j}, j \geq m+1\), are non-zero. Otherwise the arguments are simpler. By Markov's inequality  we obtain
	\begin{EQA}[c]
		\P\left(\sum_{j=m+1}^{\infty}  \xi_{j}^{2}\ge \varepsilon^{2}\right)
		\le e^{-\mu \varepsilon^{2}} \prod_{j=m+1}^{\infty} \E e^{\mu  \xi_{j}^{2}} = e^{-\mu \varepsilon^{2}} \prod_{j=m+1}^{\infty} \frac{1}{\sqrt{1-2\mu \lambda_{j}}}.
	\end{EQA}
	Choosing \(\mu \eqdef 1/(2\sum_{j=m+1}^{\infty} \lambda_{j})\) we get
	\begin{EQA}[c]
		\P\left(\sum_{j=m+1}^{\infty} \xi_{j}^{2}\ge \varepsilon^{2}\right)
		\le 2\exp\left\{-\varepsilon^{2} \left(2 \sum_{j=m+1}^{\infty}\lambda_{j}\right)^{-1} \right\}.
	\end{EQA}
	Hence, there exists \(M_{1} = M_{1}(\varepsilon)\) such that for all \(m\ge M\)
	\begin{EQA}[c]
		\P\left(\sum_{j=m+1}^{\infty} \xi_{j}^{2}\ge \varepsilon^{2}\right)\leq \varepsilon^{2}.
	\end{EQA}
	For any \( m \geq 1\) we obtain
	\begin{EQA}[c]
		\sum_{j=m+1}^{\infty} (\xi_{j} - a_{j})^{2} \le 2 \left( \sum_{j=m+1}^{\infty} \xi_{j}^{2} + \sum_{j=m+1}^{\infty} a_{j}^{2} \right).
	\end{EQA}
	We choose \( M_{2} = M_{2} (\varepsilon)\) such that \(\sum_{j=m+1}^{\infty} a_{j}^{2} \le \varepsilon^{2}\). Hence, for \( M = M_{1} + M_{2} \) we obtain the following inequality
	\begin{EQA}[c]
		\P ( x - \varepsilon \le \| \xiv - \av\|^{2} \le x + \varepsilon)  \le \P \left( x - \varepsilon - 4 \varepsilon^{2} \le \sum_{j=1}^{m} (\xi_{j} - a_{j})^{2} \le x + \varepsilon \right ) + \varepsilon^{2}. 
	\end{EQA}
	The last inequality implies 
	\begin{EQA}[c]
		\P ( x - \varepsilon \le \| \xiv - a\|^{2} \le x + \varepsilon) \le \varepsilon^{2} + (2 \varepsilon + 4\varepsilon ^{2} ) \sup_{y\in T(\varepsilon, x)} g(m, y), 
	\end{EQA}
	where \(T(\varepsilon, x) \eqdef \{y \in \mathbb{R}^{1}: x - \varepsilon -4\varepsilon^{2} \leq y \leq x + \varepsilon\}\). 
	Dividing the right-hand side of the previous inequality by \(\varepsilon\) we obtain \eqref{eq: density est} from \eqref{density for fixed m} as \(\varepsilon\) tends to 0.	
\end{proof}

\def\polhk#1{\setbox0=\hbox{#1}{\ooalign{\hidewidth
			\lower1.5ex\hbox{`}\hidewidth\crcr\unhbox0}}}

%
%
%
%
%

\end{document}